\documentclass[a4paper,reqno,11pt]{amsart}
\usepackage{amsfonts,amssymb,amsmath,amsthm,abstract,color}
\usepackage[mathscr]{euscript}
\usepackage[ps,all,arc,rotate]{xy}
\usepackage[lmargin=1in,rmargin=1in,tmargin=1in,bmargin=1in]{geometry}
\usepackage{fancyhdr}
\usepackage{pb-diagram}
\usepackage{enumitem}
\usepackage{stmaryrd}

\usepackage{hyperref}
\usepackage{cleveref}

\usepackage{bbm}

\newtheorem{prop}{Proposition}[section]
\newtheorem{lemma}[prop]{Lemma}
\newtheorem{thm}[prop]{Theorem}

\theoremstyle{definition}
\newtheorem{defn}[prop]{Definition}
\newtheorem{question}[prop]{Question}

\newtheorem{rmk}[prop]{Remark}
\newtheorem{ex}[prop]{Example}
\newtheorem{exer}[prop]{Exercise}

\newtheorem{notn}[prop]{Notation}

\DeclareMathOperator{\wt}{wt}

\DeclareMathOperator{\conv}{conv}
\DeclareMathOperator{\Int}{Int}

\DeclareMathOperator{\tr}{tr} 
        
\DeclareMathOperator{\spec}{Spec} \DeclareMathOperator{\Sym}{Sym}

\DeclareMathOperator{\Id}{Id}          
 
  \DeclareMathOperator{\im}{Im}
\DeclareMathOperator{\proj}{Proj}  
     \DeclareMathOperator{\Stab}{Stab} 
\DeclareMathOperator{\Mat}{Mat}
\DeclareMathOperator{\Gr}{Gr}     

\newcommand{\ra}{\rightarrow}      

\newcommand{\lra}{\longrightarrow}

\newcommand{\hU}{\widehat{U}}

\DeclareMathOperator{\Hom}{Hom}

\DeclareMathOperator{\Lie}{Lie}

\DeclareMathOperator{\Spec}{Spec}

\DeclareMathOperator{\Aut}{Aut}

\newcommand{\op}{\mathrm{op}}

\newcommand{\Set}{\mathrm{Set}}
\newcommand{\Sch}{\mathrm{Sch}}

\def\cA{\mathcal A}\def\cB{\mathcal B}
\def\cF{\mathcal F}\def\cG{\mathcal G}
\def\cL{\mathcal L}
\def\cM{\mathcal M}\def\cO{\mathcal O}
\def\cQ{\mathcal Q}
\def\cU{\mathcal U}\def\cV{\mathcal V}

\def\AA{\mathbb A}\def\CC{\mathbb C}
\def\GG{\mathbb G}

\def\NN{\mathbb N}\def\PP{\mathbb P}
\def\QQ{\mathbb Q}\def\RR{\mathbb R}
\def\UU{\mathbb U}
\def\ZZ{\mathbb Z}

\def\fg{\mathfrak g}
\def\fm{\mathfrak m}

\def\fu{\mathfrak u}

\def\fK{\mathfrak K}
\def\fM{\mathfrak M}

\def\fX{\mathfrak X}

 \def\GL{\mathrm{GL}} \def\SL{\mathrm{SL}}
\def\PGL{\mathrm{PGL}}
     \def\PSh{\mathrm{PSh}}

\makeatletter
\def\@tocline#1#2#3#4#5#6#7{\relax
  \ifnum #1>\c@tocdepth 
  \else
    \par \addpenalty\@secpenalty\addvspace{#2}%
    \begingroup \hyphenpenalty\@M
    \@ifempty{#4}{%
      \@tempdima\csname r@tocindent\number#1\endcsname\relax
    }{%
      \@tempdima#4\relax
    }%
    \parindent\z@ \leftskip#3\relax \advance\leftskip\@tempdima\relax
    \rightskip\@pnumwidth plus4em \parfillskip-\@pnumwidth
    #5\leavevmode\hskip-\@tempdima
      \ifcase #1
       \or\or \hskip 1em \or \hskip 2em \else \hskip 3em \fi%
      #6\nobreak\relax
    \hfill\hbox to\@pnumwidth{\@tocpagenum{#7}}\par
    \nobreak
    \endgroup
  \fi}
\makeatother

\makeatletter
\newsavebox{\@brx}
\newcommand{\llangle}[1][]{\savebox{\@brx}{\(\m@th{#1\langle}\)}%
  \mathopen{\copy\@brx\kern-0.5\wd\@brx\usebox{\@brx}}}
\newcommand{\rrangle}[1][]{\savebox{\@brx}{\(\m@th{#1\rangle}\)}%
  \mathclose{\copy\@brx\kern-0.5\wd\@brx\usebox{\@brx}}}
\makeatother

\title[Moduli spaces and GIT: old and new]{Moduli spaces and geometric invariant theory:\\ old and new perspectives}

\author{Victoria Hoskins}

\dedicatory{For Peter Newstead on his 80th Birthday}

\begin{document}

\maketitle

\begin{abstract}
Many moduli spaces are constructed as quotients of group actions; this paper surveys the classical theory, as well as recent progress and applications. We review geometric invariant theory for reductive groups and how it is used to construct moduli spaces, and explain two new developments extending this theory to non-reductive groups and to stacks, which enable the construction of new moduli spaces.
\end{abstract}

\setcounter{tocdepth}{1}
\tableofcontents

\section*{Introduction}

Mumford's geometric invariant theory (GIT) \cite{mumford} for reductive groups provides a method for constructing quotients of reductive group actions in algebraic geometry. For a reductive group acting on a projective scheme, GIT provides an open \emph{semistable} set admitting a categorical quotient which is projective and constructed from the invariant ring; furthermore, the semistable set can be explicitly described using (torus) weights via the Hilbert--Mumford criterion, rather than in terms of non-vanishing invariants.

Whilst reductive GIT has been successfully employed to construct numerous moduli spaces, it has some limitations. First, it only provides moduli spaces of \emph{semistable} objects. Second, it only applies to \emph{reductive} group actions. Third, it only applies in the situation where the moduli problem is presented in terms of a \emph{group action}. Two recent developments aim to overcome some of these issues: GIT for \emph{non-reductive groups} and \emph{stacky generalisations} of GIT. 

One of the first challenges for non-reductive group actions is the possibility of non-finitely generated invariant rings; the best-known example is Nagata's counterexample \cite{nagata} to Hilbert's 14th Problem. However, even when non-reductive invariant rings are finitely generated, the corresponding \lq GIT quotient' is not well-behaved (for example, the quotient map might not even be surjective and its image may only be constructible, see $\S$\ref{sec additive bad}). Although it is possible to construct geometric quotients of open subsets \cite{Fauntleroy,Rosenlicht,Winkelmann}, these open subsets are typically hard to describe explicitly. However, recent work on GIT  \cite{BDHK,BDHK_proj} for non-reductive groups with \emph{graded unipotent radical} (e.g.\ $\GG_a \rtimes \GG_m$ or parabolic subgroups) has enabled the construction of projective quotients of certain stable sets, which admit explicit Hilbert--Mumford type descriptions; the price to pay for obtaining these explicit projective non-reductive GIT quotients is that one must impose certain stabiliser assumptions. One of the goals of this survey is to explain the origins, assumptions and results of non-reductive GIT in as simple a context as possible to make them accessible to a broad audience, as well as to highlight some exciting applications.

We also outline another significant development that extends ideas of (reductive) GIT to stacks, as pioneered by Alper, Halpern--Leistner and Heinloth \cite{Alper_GMS,HL_Theta,AHLH,Heinloth_HMstacks}. Alper's notion of \emph{good and adequate moduli spaces} of stacks enables GIT-free constructions of moduli spaces. This is even more tangible following the recent existence criteria of Alper--Halpern-Leistner--Heinloth \cite{AHLH}, which equates the existence of moduli spaces to two simple valuative criteria and has been applied to various moduli problems \cite{AlperBelmans,ABHLX_Kstab,BayerEtAl,BDFHMT}. 

However, adequate and good moduli spaces are locally modelled on reductive GIT and require closed points to have reductive stabiliser groups. Thus, in some senses these two recent developments are orthogonal to each other and ideally there should eventually be an extension of non-reductive GIT to stacks.

\subsection*{Acknowledgements} I am indebted to both Peter Newstead and Frances Kirwan, as I learned the basics of reductive GIT from Peter Newstead's Tata lecture notes \cite{newstead} and discussions with Frances Kirwan. I am very grateful to  Greg B\'{e}rczi, Dominic Bunnett, Eloise Hamilton, Josh Jackson and Frances Kirwan for numerous conversations on non-reductive GIT. I would also like to thank the organisers of VBAC 2022 for soliciting this paper in honour of Peter Newstead.

\subsection*{Conventions} Throughout we will assume $k$ is an algebraically closed field and all schemes are assumed to be finite type $k$-schemes, unless otherwise stated. By a point, we will mean a $k$-point (or equivalently, a closed point).

\section{Moduli problems and group actions}

We start with an example-driven introduction to moduli problems in $\S$\ref{sec moduli functor} and describe the relation to group actions in $\S$\ref{sec const moduli gp}. Finally in $\S$\ref{sec gp action}, we give some basic definitions on algebraic groups, actions and quotients, which lays the foundations for GIT in $\S$\ref{sec GIT}.

\subsection{Moduli functors and spaces}\label{sec moduli functor}

Naively, a moduli problem is a collection $\cA$ of objects with an equivalence relation $\sim$ on $\cA$ and we would like to give the set of equivalence classes $\cA/\sim$ the structure of a scheme that encodes how objects vary continuously in \lq families'.

\begin{ex}\ 
\begin{enumerate}
\item Let $\cA$ be the set of $r$-dimensional linear subspaces of an $n$-dimensional $k$-vector space and $\sim$ be equality.
\item Let $\cA$ be the set of finite-dimensional $k$-vector spaces with an endomorphism and $\sim$ be vector space isomorphisms commuting with the endomorphism.
\item Let $\cA$ be the set of $n \times n$ matrices over $k$ and $\sim$ be the equivalence relation given by similarity of matrices. 
\item Let $\cA$ to be the set of hypersurfaces of degree $d$ in $\PP^n$ and $\sim$ be the relation given by projective change of coordinates.
\item Let $\cA$ be the collection of smooth projective curves of fixed genus and $\sim$ be the relation given by isomorphism.
\item Let $\cA$ be the collection of vector bundles on a fixed scheme $X$ and $\sim$ be the relation given by isomorphisms of vector bundles.
\end{enumerate}
\end{ex}

Often there is a natural notion of families of objects over a scheme $S$ and an extension of $\sim$ to families over $S$, such that we can pullback families by morphisms $T \ra S$ compatibly with the notion of equivalence.

\begin{ex}\ 
\begin{enumerate}
\item A family over $S$ of $r$-dimensional linear subspaces of an $n$-dimensional vector space is a rank $r$ vector subbundle $\cV \subset \cO_S^{\oplus n}$.
\item A family over $S$ of vector spaces with endomorphisms is a vector bundle $\cV$ over $S$ with an endomorphism $\Phi: \cV \ra \cV$.
\end{enumerate}
\end{ex}

The next example shows there might be several ways to extend $(\cA,\sim)$ to families over $S$.

\begin{ex}
For vector bundles on a fixed scheme $X$ up to isomorphism, the natural notion for a family over $S$ is a vector bundle $\cF$ over $X \times S$ over $S$, but there are at least two natural equivalence relations:
\[\begin{array}{ccl} \cF \sim'_S \cG & \iff & \cF \cong \cG \\ \cF \sim_S \cG &\iff & \cF \cong \cG \otimes \pi_S^* \cL \mathrm{\:\:for \:\: a \: \: line \:\: bundle \: \:} \cL \ra S, \end{array}\]
where $\pi_S : X \times S \ra S$. Since $\cL \ra S$ is locally trivial, there is a cover $S_i$ of $S$ such that $\cF|_{X \times S_i} \cong \cG|_{X \times S_i}$. Hence $\sim_S$ can be thought of a Zariski local version of $\sim'_S$. 
\end{ex}

We can now give a more precise definition of a moduli problem using families.

\begin{defn}[Moduli problem and moduli functor]
A \emph{moduli problem} consists of
\begin{enumerate}
\item for each scheme $S$, a collection $\cA_S$ of families over $S$ with an equivalence relation $\sim_S$, 
\item for each morphism $T \ra S$ of schemes, a pullback map $f^* : \cA_S \ra \cA_T$ 
\end{enumerate}
such that
\begin{enumerate}
\renewcommand{\labelenumi}{\: \: (\roman{enumi})}
\item for $f: T \ra S$ and equivalent families $\cF \sim_S \cG$ over $S$, we have $f^* \cF \sim_T f^* \cG$;
\item for any family $\cF$ over $S$, we have $\text{Id}_S^*\cF = \cF$;
\item for any morphisms $f: T \ra S$ and $g: S \ra R$, and a family $\cF$ over $R$, we have an equivalence $(g \circ f)^*\cF \sim_T f^* g^* \cF$.
\end{enumerate}
This gives rise to a \emph{moduli functor} $\cM : \Sch^{\op} \ra \Set$ where 
\[\cM(S) := \{\mathrm{families \:\:over\: } \: S\}/\sim_S \quad \text{and} \quad \cM(f: T \ra S)=f^* : \cM(S) \ra \cM(T).\]
\textbf{Notation:} For a family $\cF$ over $S$ and a point $s : \spec k \ra S$, we write $\cF_s :=s^*\cF$ to denote the corresponding family over $\spec k$. We write $(\cA,\sim) :=(\cA_{\spec k},\sim_{\spec k})$.
\end{defn}

In particular, a moduli functor is a presheaf on the category $\Sch$ of schemes. Recall that the Yoneda Lemma gives an embedding of the category of schemes into the category of presheaves; more precisely there is a fully faithful functor $h : \Sch \ra \PSh(\Sch)$ which on objects sends a scheme $X$ to its functor of points $\Hom(-, X) : \Sch^{\op} \ra \Set$. A presheaf is called \emph{representable} if it is in the essential image of the Yoneda embedding. 

A moduli functor being representable is the ideal situation and leads to the notion of a fine moduli space, but if that fails, one can instead ask for a universal natural transformation from $\cM$ to the functor of points of a scheme, which leads to the notion of a coarse moduli space. 

\begin{defn}[Fine and coarse moduli spaces]
Let $\cM: \Sch \ra \Set$ be a moduli functor.
\begin{enumerate}[label={\roman*)}]
\item A scheme $M$ is a \emph{fine moduli space} for $\cM$ if it 
represents $\cM$; that is, there is a natural isomorphism $\cM \ra \Hom(-, M)$. In this case, $\Id_M \in \Hom(M,M)$ corresponds to an element of $\cM(M)$ called the \emph{universal family} $\cU$, which is a family over $M$ up to the notion of equivalence.
\item A \emph{coarse moduli space} for $\cM$ is a scheme $M$ with a natural transformation of functors $\eta: \cM \ra h_M$ which is universal (for any natural transformation $\nu: \cM \ra \Hom(-,N)$ to the functor of points of a scheme, there exists a unique morphism $f: M \ra N$ such that $\nu = f_* \circ \eta$) such that $\eta_{\spec k} : \cM(\spec k) \ra h_M(\spec k)$ is bijective.
\end{enumerate}
\end{defn}

\begin{ex}
For 1-dimensional subspaces of $k^n$, a fine moduli space is given by $\PP^n$, with the tautological line bundle $\cO_{\PP^n}(-1) \subset \cO^{\oplus n}_{\PP^n}$ giving a universal family (see \cite[II Theorem 7.1]{hartshorne}).
\end{ex}

\begin{rmk}\
\begin{enumerate}
\item If a fine or coarse moduli space exists, then it is unique up to unique isomorphism.
\item If a fine moduli space exists, then the universal family $\cU$ over $M$ describes all other families in the following sense: for any scheme $S$, we have that a family $\cF \in \cM(S)$ is equivalent to a morphism $f : S \ra M$ with $f^*\cU \sim_S \cF$.
\item Since $\Hom(-,M)$ is a sheaf in the Zariski toplogy (that is, 
for every scheme $S$ and Zariski cover $\{S_i\}$ of $S$, the natural map
\[\{f \in  F(S)\} \longrightarrow \{ (f_i \in F(S_i))_i : f_i|_{S_i \cap S_j} =f_j|_{S_j \cap S_i} \text{ for all } i,j \}\] 
is a bijection), for a moduli functor $\cM$ to admit a fine moduli space it must be a sheaf in the Zariski topology.
\end{enumerate}
\end{rmk}

\begin{exer}
Show that the equivalence relation $\sim'_S$ on families over $S$ of vector bundles on a fixed scheme defines a moduli functor that is not representable. (Hint: Show it fails to be a Zariski sheaf by considering the other equivalence relation $\sim_S$ for families of vector bundles).
\end{exer}

Unfortunately, there may be moduli problems which do not admit even a coarse moduli space.

\begin{exer}
Let $\cM$ be a moduli functor with the \emph{jump phenomenon}; that is, there is a family $\cF$ over $\AA^1$ such that $\cF_s \sim \cF_1$ for all $s \neq 0$ and $\cF_0 \nsim \cF_1$. Show that there is no coarse moduli space for $\cM$ by showing for any natural transformation $\eta : \cM \ra \Hom(-,M)$, the morphism $\eta_{\AA^1}(\cF): \AA^1 \ra M$ is constant.  
\end{exer}

\begin{ex}
Moduli of rank 2 degree 0 vector bundles on $\PP^1$ exhibit the jump phenomenon: there is a family $\cF$ of rank 2 degree 0 vector bundles over $\AA^1$ such that
\[ \cF_s = \left\{ \begin{array}{ll} \cO_{\PP^1}^{\oplus 2} & s \neq 0 \\ \cO_{\PP^1}(1) \oplus \cO_{\PP^1}(-1) & s = 0. \end{array} \right.\]
Indeed this family is constructed using the isomorphisms
\[ \text{Ext}^1(\cO_{\PP^1}(1), \cO_{\PP^1}(-1)) \cong H^1(\PP^1, \cO_{\PP^1}(-2)) \cong H^0(\PP^1,\cO_{\PP^1})^* \cong k.\]
\end{ex}

Another reason for a coarse moduli space to fail to exist is if the moduli problem is \emph{unbounded}: there does not exist a family $\cF$ over a scheme $S$ (of finite type over $k$) such that for any object $E$ (i.e family over $k$), we have $E \sim \cF_s$ for some (possibly non-unique) $s \in S$.

\begin{exer}
Show the moduli problem of rank 2 degree 0 vector bundles on $\PP^1$ is unbounded: suppose there exists a family $\cF$ over $S$ such that for any such vector bundle $E$ we have $\cF_s \sim E$ for some $s \in S$, then show that $S$ cannot be Noetherian by considering the subschemes 
\[ S_n:=\{ s \in S : \dim H^0(\PP^1,\cF_s) \geq n\},\]
which are closed by the semi-continuity theorem.
\end{exer}

A further obstruction to the existence of a coarse moduli space is that there may be non-trivial families which are fibrewise trivial; this may happen when there are non-trivial automorphisms.

One possible solution for dealing with some of these issues is to instead work with a moduli stack; in this case, one can then look for a coarse or good moduli space for this stack as in $\S$\ref{sec stacks beyond GIT}. However, often by imposing a notion of stability on objects, which can be moduli-theoretic or arising from the GIT construction, one obtains a much better-behaved moduli problem for which one can construct moduli spaces.

\subsection{Construction of moduli spaces using group actions}\label{sec const moduli gp}

Many moduli spaces are constructed as quotients of group actions via the following strategy (after fixing any discrete invariants and restricting to a bounded class of objects):
\begin{enumerate}
\item Find an overparametrisation: find a parameter scheme $X$ with a family $\cF$ such that any other family can be locally obtained by pullback from $\cF$ (possibly non-uniquely).
\item Find a group action describing the symmetries: find a group $G$ acting on $X$ such that the orbits correspond to the equivalence classes.
\item Take a quotient (in the category of schemes if possible).
\end{enumerate}
The third step is typically performed using Geometric Invariant Theory (GIT).

Let us give some examples, before describing algebraic group actions in $\S$\ref{sec gp action} in more detail.

\begin{ex}\label{ex moduli gp actions}\
\begin{enumerate}
\item 1-dimensional vector subspaces in $V = k^n$ can be parametrised by fixing a basis vector in $X = V \setminus \{ 0 \}$. Since two basis vectors are related by scalar multiplication, $\GG_m$ acting on $X$ describes the symmetries and $\PP^{n-1} = X /\GG_m$ is the fine moduli space.
\item An $r$-dimensional subspace in $V= k^n$ can be parametrised by choosing a basis, which gives an element in $\Mat_{r \times n}$ of rank $r$ and the choice of basis is controlled by the action of $\GL_r$ by left multiplication. In Exercise \ref{exer grassmannain}, we will see that the open locus of rank $r$ matrices is a GIT semistable locus that admits a quotient, namely a Grassmannian.
\item A projective hypersurface of degree $d$ in $\PP^n$ is given by the vanishing locus $\{ F = 0 \}$ of a degree $d$ homogeneous polynomial in $n+1$ variables. Since $\lambda F$ defines the same hypersurface, one can consider $X = \PP(k[x_0,\dots,x_n]_d)$ and the action of $G = \PGL_n = \Aut(\PP^n)$ describes these hypersurfaces up to change of coordinates.
\item\label{ex moduli 4} For vector bundles of rank $n$ and degree $d$ on a smooth projective curve $C$, provided $d$ is sufficiently large, any \emph{semistable}\footnote{There is a natural notion of semistability involving verifying an inequality of slopes for all subbundles, which turns out to be related to a corresponding GIT notion of semistability.} vector bundle $E$ can be parametrised as a quotient of a fixed vector bundle (see \cite[Lemma 5.2]{newstead}): for $E$ semistable of sufficiently large degree, the evaluation map is surjective and $E$ has vanishing higher cohomology, so by choosing a basis of global sections we obtain a quotient
\[ \cO_C^{\oplus \chi} \cong H^0(E) \otimes \cO_C \stackrel{\mathrm{ev}}{\twoheadrightarrow} E \] 
where $\chi = d + n (1-g)$ is the Euler characteristic of $E$. Consequently, a Quot scheme parametrising quotients of $\cO_C^{\oplus \chi}$ with fixed invariants gives an overparametrisation and the action of $\GL_\chi$ describes the symmetries. 
\end{enumerate}
\end{ex}

\subsection{Algebraic groups, actions and quotients}\label{sec gp action}

Here we focus on the essential notions that we will need, and refer to \cite{borel,brion,milneAGS} for more detailed expositions. 

\begin{defn}
An \emph{algebraic group} over $k$ is a a group object in the category of $k$-schemes (that is, a $k$-scheme $G$ with identity element $e : \Spec k \ra G$, group operation $m : G \times G \ra G$ and inversion $i : G \ra G$ given by morphisms of schemes such that the usual group axioms are stated as commutativity of certain diagrams). We say $G$ is an \emph{affine algebraic group} if the underlying scheme $G$ is affine.
\end{defn}

\begin{rmk}
The $k$-algebra $\cO(G)$ of regular functions on $G$ is a \emph{Hopf algebra} with comultiplication $m^* : \cO(G) \ra \cO(G) \otimes \cO(G)$, coinversion $i^*: \cO(G) \ra \cO(G)$ and counit $e^*: \cO(G) \ra k$ and dualised commutative diagrams. In fact, there is a one-one correspondence between finitely generated Hopf algebras over $k$ and affine algebraic groups over $k$  \cite[II Theorem 5.1]{milneAGS}.
\end{rmk}

As the following example demonstrates, many familiar groups are affine algebraic groups.

\begin{ex}\
\begin{enumerate}
\item The \emph{additive group} $\GG_a = \spec k[t]$ over $k$ is the algebraic group whose underlying scheme is the affine line $\AA^1$ over $k$ and whose group operation is given by addition:
\[ m^*(t) = t \otimes 1 + 1 \otimes t \quad \mathrm{and} \quad i^*(t) = - t .\]
For a $k$-algebra $R$, we have $\GG_a(R)=(R, +)$.
\item The \emph{multiplicative group} $\GG_m = \spec k[t,t^{-1}]$ over $k$ is the algebraic group whose underlying variety is the $\AA^1 - \{0\}$ and whose group operation is given by multiplication:
\[ m^*(t) = t \otimes t \quad \mathrm{and} \quad i^*(t) = t^{-1} .\] For a $k$-algebra $R$, we have $\GG_m(R)=(R^\times, \cdot)$.
\item The \emph{general linear group} $\GL_n$ over $k$ is an open subvariety of $\AA^{n^2}$ cut out by the non-vanishing of the determinant. It is an affine variety with coordinate ring  $k[x_{ij} : 1 \leq i,j \leq n]_{\det(x_{ij})}$. The co-group operations are defined by:
\[ m^*(x_{ij}) = \sum_{k=1}^n x_{ik} \otimes x_{kj} \quad \mathrm{and} \quad i^*(x_{ij}) =(x_{ij})^{-1}_{ij} \]
where $(x_{ij})^{-1}_{ij}$ is the regular function on $\GL_n$ given by taking the $(i,j)$-th entry of the inverse of a matrix. 
\item For a finite group $G$, the group algebra $k[G]$ is a Hopf algebra and determines an affine algebraic group $\underline{G}_k:=\spec(k[G])$, whose $k$-points are identified with elements of $G$. 
\item For $n\geq 1$, the group of $n$th roots of unity is $\mu_n:=\spec k[t,t^{-1}]/(t^n-1)\subset \GG_m$. Write $I$ for the ideal $(t^n-1)$ of $R:=k[t,t^{-1}]$. Then 
\[m^*(t^n-1) = t^n\otimes t^n - 1\otimes 1 = (t^n-1)\otimes t^n + 1\otimes (t^n-1)\in I\otimes R+ R\otimes I\]
which implies that $\mu_n$ is an algebraic subgroup of $\GG_m$. If $n$ is different from $\mathrm{char}(k)$, the polynomial $X^n-1$ is separable and there are $n$ distinct roots in $k$. Then the choice of a primitive $n$th root of unity in $k$ determines an isomorphism $\mu_n\simeq \underline{\ZZ/n\ZZ}_{\: k}$. However, if $n=\mathrm{char}(k)$, then $X^n-1=(X-1)^n$ in $k[X]$, which implies that the scheme $\mu_n$ is non-reduced (with $1$ as the only closed point).
\end{enumerate}
\end{ex}

A \emph{linear algebraic group} is a closed subgroup of $\GL_n$; hence, any linear algebraic group is an affine algebraic group. The converse statement is also true: any affine algebraic group is a linear algebraic group (see Remark \ref{rmk aff alg gp is lag}).

\begin{defn}
An (algebraic) \emph{action} of an affine algebraic group $G$ on a scheme $X$ is a morphism of schemes $\sigma : G \times X \rightarrow X$ such that the following diagrams commute
\[ \xymatrix@C=2.5em{ \spec k \times X \ar[r]^{\quad e \times \text{id}_X} \ar[rd]_{\cong} & G \times X \ar[d]^{\sigma} & 
& G \times G \times X \ar[r]^{\quad \text{id}_G \times \sigma} \ar[d]_{m_G \times \text{id}_X} & G \times X \ar[d]^{\sigma} \\
& X & & G \times X \ar[r]_{\sigma} & X.  }
\] 
A subscheme $Z \subset X$ is called \emph{$G$-invariant} if it is preserved by the action; that is, $\sigma( G \times Z) \subset Z$.

A morphism $f : X \ra Y$ between schemes with actions $\sigma_X: G \times X \ra X$ and $\sigma_Y : G \times Y \ra Y$ is \emph{$G$-equivariant} if the following diagram commutes
\[ \xymatrix@C=2.5em{ G \times X \ar[r]^{\text{id}_G \times f} \ar[d]_{\sigma_X} & G \times Y \ar[d]_{\sigma_Y} \\ X \ar[r]_{f} & Y. }
\]  
If $Y$ is given the trivial action $\sigma_Y =\pi_Y : G \times Y \ra Y$, then we say $f : X \ra Y$ is \emph{$G$-invariant}.
\end{defn}

\begin{rmk}
For an action $\sigma: G \times X \ra X$, there is an induced action of $G$ on the ring of regular functions $\cO(X)$ given by
\[ (g \cdot f)(x) = f(g^{-1} \cdot x)\]
for $g \in G$, $f \in \cO(X)$ and $x \in X$.
\end{rmk}

By the following lemma, any action of an affine algebraic group $G$ on an affine scheme $X$ gives a $G$-action on the $k$-algebra $\cO(X)$ which is \emph{rational}; that is, every $f \in \cO(X)$ is contained in a finite dimensional $G$-invariant linear subspace of $\cO(X)$.

\begin{lemma}\label{fdim lemma}
For an affine algebraic group $G$ acting on an affine scheme $X$, any finite dimensional vector subspace of $\cO(X)$ is contained in a finite dimensional $G$-invariant vector subspace.
\end{lemma}
\begin{proof}
Let $\sigma^* : \cO(X) \ra \cO(G) \otimes \cO(X)$ denote the coaction. Let $W = \mathrm{Span}_k(f_1,\dots,f_n) \subset \cO(X)$ and write $\sigma^*(f_i) = \sum_{j=1}^{n_i} h_{ij} \otimes f_{ij}$ with $h_{ij} \in \cO(G)$ and $g_{ij} \in \cO(X)$. The vector space spanned by $f_{ij}$ is a $G$-invariant finite-dimensional subspace containing $W$, as $ g \cdot f_i = \sum_j h_{ij}(g) f_{ij}$.
\end{proof}

\begin{rmk}\label{rmk aff alg gp is lag}
By applying this to the action of $G$ on itself by left multiplication, if we let $W$ be a vector space spanned by a finite choice of algebra generators for $\cO(G)$, then $W$ is contained in a finite dimensional $G$-invariant vector subspace $V \subset \cO(X)$. One can prove that there is an embedding $G \ra \GL(V)$ to show any affine algebraic group over $k$ is a linear algebraic group.
\end{rmk}

We can define orbits and stabilisers in this setting; the latter has a scheme structure by definition, and we shall soon see the former can also be equipped with a scheme structure.

\begin{defn}[Orbits and stabilisers]
For an action $\sigma: G \times X \ra X$ of an affine algebraic group $G$ on a scheme $X$ and for a $k$-point $x \in X$, we define
\begin{enumerate}
\item the \emph{orbit} $G \cdot x$ of $x$ to be the (set-theoretic) image of $\sigma_x = \sigma(-,x) : G(k) \ra X(k)$ given by $g \mapsto g \cdot x$;
\item the \emph{stabiliser} $G_x$ of $x$ to be the fibre product of $\sigma_x : G \ra X$ and $x: \spec k \ra X$. 
\end{enumerate}
\end{defn}

The stabiliser $G_x$ of $x$ is a closed subscheme of $G$ (as it is the preimage of a closed subscheme of $X$ under $\sigma_x : G \ra X$) and a subgroup of $G$. By Chevalley's Theorem \cite[II Exercise 3.19]{hartshorne}, as the image of a morphism of schemes, the orbit is a priori only a constructible subset of $X$. However, we claim it is a locally closed subset and so can be equipped with the structure of a reduced locally closed subscheme of $X$. Indeed, $G\cdot x$ is open in its closure: since the orbit is constructible, there is a dense open subset $U$ of $\overline{G\cdot x}$ with $U\subset G\cdot x$ and, as $G$ acts transitively on the orbit, every point in the orbit is contained in a $G$-translate of $U$. 

The boundary of an orbit is a union of orbits of strictly smaller dimension, and so in particular every orbit closure contains a closed orbit (of minimal dimension). There is also an orbit stabiliser theorem:
\[ \dim G = \dim G_x + \dim G \cdot x\]
as $\sigma_x:G\rightarrow G\cdot x$ is flat (by transitivity of the $G$-action, we can deduce this from generic flatness) and so we can apply the dimension formula for fibres of a flat morphism \cite[Proposition III.9.5]{hartshorne}.

In general, for an action $\sigma : G \times X \ra X$, the set of orbits $X/G$ is not a scheme. Instead, we ask for a universal quotient in the category of schemes.

\begin{defn}[Categorical quotient]
For an action of an affine algebraic group $G$ on scheme $X$, a \emph{categorical quotient} is a $G$-invariant morphism $\varphi : X \ra Y$ of schemes which is universal (that is, every other $G$-invariant morphism $f : X \ra Z$ factors uniquely through $\varphi$ so that there exists a unique morphism $h : Y \ra Z$ such that $f = h \circ \varphi$). If the preimage of each $k$-point in $Y$ is a single orbit, then we say $\varphi$ is an \emph{orbit space}.
\end{defn}

If a categorical quotient exists, then it is unique up to unique isomorphism. In cases where a categorical quotient does not exist, one may want to enlarge the category of schemes to algebraic spaces or even algebraic stacks.

A categorical quotient is constant on orbits and orbit closures. Hence, a categorical quotient is an orbit space only if the action of $G$ on $X$ is closed; that is, all the orbits $G \cdot x$ are closed.

\begin{ex}\
\begin{enumerate}
\item For $\GG_m$ acting on $\AA^n$ by scalar multiplication $t \cdot (a_1, \dots, a_n) = (t a_1, \dots, t a_n)$, there are two types of orbits:
\begin{itemize}
\item punctured lines through the origin,
\item the origin (closed of dimension $0$)
\end{itemize}
Every orbit contains the origin in its closure. As any $\GG_m$-invariant function on $\AA^n$ is constant on orbits and their closures, it must be constant and so factors via the structure map $\pi : \AA^n \ra \Spec k$. Hence, the structure map is a categorical quotient.
\item For the action of $\GG_m$ on $\AA^{2}$ by $t \cdot (x,y) = (tx, t^{-1}y)$, the orbits are
\begin{itemize}
\item conics $\{(x,y): xy = \alpha\}$ for $\alpha \in \AA^1 \setminus \{0\}$ (closed of dimension $1$),
\item the punctured $x$-axis,
\item the punctured $y$-axis,
\item the origin (closed of dimension $0$).
\end{itemize}
The punctured axes both contain the origin in their orbit closures. We will see that the categorical quotient for this action is $\AA^2 \ra \AA^1$ given by $(x,y) \mapsto xy$.
\end{enumerate}
\end{ex}

We see the sort of problems that may occur when we have non-closed orbits. In the first example, our geometric intuition tells us that we would ideally like to remove the origin and then take the quotient of $\GG_m$ acting on $\AA^n \setminus \{ 0 \}$ to obtain the projective space $\PP^{n-1} = (\AA^n \setminus \{0\})/ \GG_m$, which is an orbit space for this action. We will return to this example and see that by introducing a non-trivial notion of GIT semistability, we can remove the origin (see Example \ref{ex Pn as twisted affine quotient}).

There is the following stronger notion of quotient that arises in GIT \cite[Definition 1.5]{seshadri_quotients}.

\begin{defn}[Good quotient]\label{good}
A morphism $\varphi : X \ra Y$ is a \emph{good quotient} for an action of $G$ on $X$ if
\begin{enumerate}
\renewcommand{\labelenumi}{\roman{enumi})}
\item $\varphi$ is $G$-invariant, surjective and affine;
\item The map $\cO_Y \ra \varphi_* \cO_X^G$ is an isomorphism.
\item If $W_1$ and $W_2$ are disjoint $G$-invariant closed subschemes, then $\varphi(W_1) $ and $\varphi(W_2)$ are disjoint closed subschemes.
\end{enumerate}
If moreover the preimage of each point is a single orbit, then we say $\varphi$ is a \emph{geometric quotient.}
\end{defn}

\begin{rmk}\label{rmk conseq good}\
\begin{enumerate}
\item The definition of a good quotient is local in the target, which enables the construction of good quotients via gluing. 
\item The last two conditions imply that $\varphi$ is surjective: the second property shows that $\varphi$ is dominant (i.e.\ the image of $\varphi$ is dense in $Y$) and the third condition shows that the image of $\varphi$ is closed. Furthermore it implies that for each $y \in Y$, the preimage $\varphi^{-1}(y)$ contains a unique closed orbit, whose stabiliser is reductive (see Definition \ref{def red} and \cite{luna}, as the quotient of the reductive group $G$ by the subgroup $G_x$ is affine if and only if $G_x$ is reductive). In particular, if all orbits are closed, then $\varphi$ is a geometric quotient.

\item The third condition also enables us to determine when two orbit closures meet: we have $\overline{G \cdot x_1} \cap \overline{G \cdot x_2} \neq \phi$ if and only if $\varphi(x_1) = \varphi(x_2)$.
\item Any good quotient is a categorical quotient; see \cite[Proposition 3.11]{newstead}.
\end{enumerate}
\end{rmk}

Let us relate the construction of moduli spaces with categorical quotients. For a moduli problem $\cM$, a family $\cF$ over a scheme $S$ has the \emph{local universal property} if for any other family $\cG$ over a scheme $T$ and for any $k$-point $t \in T$, there exists a neighbourhood $U$ of $t$ in $T$ and a morphism $f: U \ra S$ such that $\cG|_U \sim_U f^*\cF$. 
 
\begin{prop}\label{prop moduli cat quot} \emph{\cite[Proposition 2.13]{newstead}}
Let $\cM$ be a moduli problem for which there exists a family $\cF$ over $S$ with the local universal property. Suppose that there is an algebraic group $G$ acting on $S$ such that two $k$-points $s,t$ lie in the same $G$-orbit if and only if $\cF_t \sim \cF_s$. Then
\begin{enumerate}
\item any coarse moduli space is a categorical quotient of the $G$-action on $S$;
\item a categorical quotient of the $G$-action on $S$ is a coarse moduli space if and only if it is an orbit space.
\end{enumerate}
\end{prop}
\begin{proof}
For any scheme $M$, we claim that there is a bijective correspondence
\[\{\text{natural transformations } \eta: \cM \ra \Hom(-,M)\}  \longleftrightarrow \{G\text{-invariant morphisms } f: S \ra M\}\]
given by $\eta \mapsto \eta_S(\cF)$, which is $G$-invariant by our assumptions about the $G$-action on $S$. Conversely, given a $G$-invariant morphism $f: S \ra M$, we define $\eta: \cM \ra \Hom(-,M)$ to associate to a family $\cG$ over $T$, a morphism $\eta_T(\cG) : T \ra M$ glued together locally by using the local universal property of $\cF$ over $S$. More precisely, we can cover $T$ by open subsets $U_i$ such that there is a morphism $h_i : U_i \ra S$ and $h_i^* \cF \sim_{U_i} \cG|_{U_i}$. For $u \in U_i \cap U_j$, we have 
\[\cF_{h_i(u)} \sim (h_i^*\cF)_{u} \sim \cG_u \sim (h_j^*\cF)_{u} \sim \cF_{h_j(u)}\]
and so by assumption $h_i(u)$ and $h_j(u)$ lie in the same $G$-orbit. Since $f$ is $G$-invariant, the compositions $f \circ h_i : U_i \ra M$ glue to a morphism $\eta_T(\cG) : T \ra M$. 

Hence, if $(M,\eta: \cM \ra h_M)$ is a coarse moduli space, then $\eta_S(\cF) : S \ra M$ is $G$-invariant and the universal $G$-invariant morphism from $S$, which proves statement a). Furthermore, the $G$-invariant morphism $\eta_S(\cF) : S \ra M$ is an orbit space if and only if $\eta_{\spec k}$ is bijective, which proves statement b).
\end{proof}

\section{Mumford's reductive geometric invariant theory}\label{sec GIT}

The origins of GIT go back to 19th century invariant theory and a question of Hilbert on the finite generation of invariant rings. We begin with Hilbert's 14th problem in $\S$\ref{sec Hilb14} and describe other techniques for constructing quotients in $\S$\ref{sec other methods const quotients}. We give various definitions and examples related to reductive and unipotent groups in $\S$\ref{sec reductive}, and prove finite generation results for invariant rings in $\S$\ref{sec fg inv}. We describe Mumford's GIT \cite{mumford} for affine schemes in $\S$\ref{sec affine GIT}, for projective schemes in $\S$\ref{sec proj GIT} and in general in $\S$\ref{sec general GIT}, and state a twisted affine version in $\S$\ref{sec affine GIT twisted char}. Beyond Mumford's book \cite{mumford}, the notes of Newstead \cite{newstead} and Thomas \cite{thomas} provide excellent introductions to GIT.

\subsection{Hilbert's 14th Problem}\label{sec Hilb14}

Consider an action $\sigma : G \times X \ra X$ of an affine algebraic group on an affine scheme. The coaction determines a linear representation $G \ra \GL(\cO(X))$. Any $G$-invariant morphism $\phi : X \ra Z$ induces a homomorphism $\phi^* : \cO(Z) \ra \cO(X)$ whose image is contained in the ring of $G$-invariant functions
\[ \cO(X)^G := \{ f \in \cO(X) : g \cdot f = f \: \text{ for all } g \in G \}.\] 
Consequently one can ask if the inclusion of the ring of $G$-invariant functions corresponds to a morphism of affine schemes. 

\begin{question}[Hilbert's 14th Problem] Is $\cO(X)^G$ a finitely generated $k$-algebra?
\end{question}

Hilbert showed that for the general linear group over the complex numbers, the answer was yes. However, in general $\cO(X)^G$ is not finitely generated, due to a counterexample of Nagata constructed using an action of a product of additive groups \cite{nagata59,nagata}; see \cite{Freudenburg_Hilbert14} for a survey of counterexamples to Hilbert's 14th problem. Fortunately, Nagata also showed that the answer is yes for a large number of groups, namely reductive groups and in this case Mumford showed that taking the spectrum of the inclusion $\cO(X)^G \subset \cO(X)$ gives a categorical quotient of the action. For non-reductive groups, we will see in $\S$\ref{sec additive bad} that even when $\cO(X)^G$ is finitely generated, taking the spectrum of the inclusion of invariants does not always yield a categorical quotient.

\subsection{Constructions of quotients by affine algebraic groups}\label{sec other methods const quotients}

Before turning to Mumford's GIT for reductive groups, let us give a brief summary of some important results about the construction of quotients of affine algebraic groups in general.

For a free action of an affine algebraic group on a scheme, there is a geometric quotient in the category of \emph{algebraic spaces} by results of Artin \cite{Artin} and Koll\'{a}r \cite{Kollar}. For a finite group, it is easy to see that a free action induces an \'{e}tale equivalence relation and any quotient of a scheme by an \'{e}tale equivalence relation is an algebraic space. There are examples of free actions whose geometric quotient is not a scheme: an example of Hironaka gives an action of a finite group whose geometric quotient is not a scheme (see \cite[Appendix B, Example 3.4.1]{hartshorne}) 
and an example of Derksen gives an action of the additive group $\GG_a$ whose geometric quotient is not a scheme (see \cite[Example 18]{DerksenFreeGa}),

Rosenlicht \cite{Rosenlicht} showed that for a connected affine algebraic group  $G$ acting on an irreducible variety $X$, there is a dense open subset which admits a geometric quotient. Unfortunately, this set is non-explicit, as his proof involves showing that the field $k(X)^G$ of invariant \emph{rational} functions is finitely generated (see \cite[$\S$19 Appendix Lemma 1]{GrosshansBook}).

\begin{rmk}[Transfer Principle]\label{rmk transfer principle}
Let $H < G$ be a closed subgroup of an affine algebraic group $G$, then $H$ acts on $G$ by left multiplication and this action has a geometric quotient $G/H$ and the quotient map $G \ra G/H$ is \'{e}tale locally trivial (see \cite[p181]{borel} and \cite[Theorem 1.16]{brion}).

For $H < G$ as above, suppose there is an action of $H$ on $X$; then for the diagonal action of $H$ on $G \times X$ by $h \cdot (g,x) = (gh^{-1},hz)$, there is a geometric quotient $G \times^HX$; for details on this construction, see \cite[III $\S$4]{DemazureGabriel}. The action of $G$ on itself by left multiplication induces a $G$-action on $G \times^H X$ such that there is a bijective correspondence between $H$-orbits in $X$ and $G$-orbits in $G \times^H X$. A scheme is a geometric $H$-quotient of $X$ if and only if it is a geometric $G$-quotient of $G \times^HX$. Furthermore, if $X$ is affine and the $H$-action on $X$ extends to $G$, then we have the following transfer principle (due to Roberts, see \cite[$\S$9]{GrosshansBook})
\[ \cO(X)^H \cong (\cO(X) \otimes \cO(G/H))^G.\]
Note $G/H$ may not be affine, so $ \cO(G/H)$ is not necessarily finitely generated (see Remark \ref{rmk Grosshans subgp}).
\end{rmk}

\subsection{Reductive groups}\label{sec reductive} Let us fix some definitions for unipotent and reductive groups; we note that there are alternative equivalent formulations (for example, see \cite{conrad,milneAGS}).

\begin{defn}[Unipotent and reductive groups]\label{def red} An affine algebraic $k$-group $G$ is 
\begin{enumerate}[label={\roman*)}]
\item \emph{unipotent} if it is isomorphic to a subgroup of a standard unipotent group $\UU_n \subset \GL_n$ consisting of upper triangular matrices with diagonal entries equal to 1.
\item \emph{reductive} if it is smooth and every connected unipotent normal subgroup is trivial (this second condition is often phrased as asking for the unipotent radical to be trivial).
\item \emph{geometrically reductive} if for every finite dimensional linear representation $\rho : G \ra \GL(V)$ and every non-zero $G$-invariant point $v \in V$, there is a non-constant $G$-invariant homogeneous polynomial $f \in \cO(V)$ such that $f(v) \neq 0$.
\item \emph{linearly reductive} if for every finite dimensional linear representation $\rho : G \ra \GL(V)$ and every non-zero $G$-invariant point $v \in V$, there is a non-constant $G$-invariant linear polynomial $f \in \cO(V)$ such that $f(v) \neq 0$.
\end{enumerate}
\end{defn}

\begin{rmk}\
\begin{enumerate}
\item $G$ is unipotent if and only if every finite dimensional linear representation $\rho : G \ra \GL(V)$ has a non-zero fixed point.
\item G is linearly reductive if and only if every finite dimensional linear representation $\rho : G \ra \GL(V)$ is completely reducible (that is, $\rho$ decomposes as a direct sum of irreducible representations) or equivalenty if taking $G$-invariants on finite dimensional linear $G$-representations is exact.

\end{enumerate}

\end{rmk}

\begin{ex}\
\begin{enumerate}
\item The additive group $\GG_a$ is unipotent, as we have an embedding $\GG_a \hookrightarrow \UU_2$ given by
\[ c \mapsto \left( \begin{array}{cc} 1 & c \\ 0 & 1\end{array} \right).\]
\item In characteristic $p$, there is a finite subgroup $\alpha_p \subset \GG_a$ where we define the functor of points of $\alpha_p$ by associating to a $k$-algebra $R$, 
\[ \alpha_p(R):= \{ c \in \GG_a(R) : c^p=0 \}. \]
This is represented by the scheme $\spec k[t]/(t^p)$ and so $\alpha_p$ is a unipotent group which is not smooth.
\item The multiplicative group $\GG_m$ or any algebraic torus $T=\GG_m^r$ is linearly reductive: as any linear representation $\rho : T \ra \GL(V)$ admits a weight decomposition
\[ V = \bigoplus_{\chi \in X^*(T)}  V_\chi \quad \text{ where } \: V_{\chi} = \{ v \in V : t \cdot v = \chi(t)v \: \text{ for all } \: t \in T \}.\]
\end{enumerate}
\end{ex}

\begin{exer}
Prove that any finite group of order not divisible by the characteristic of $k$ is linearly reductive. 
(Hint: consider averaging over the group.)
\end{exer}

For smooth affine algebraic group schemes over $k$, we have
\[ \text{linearly reductive} \implies \text{geometrically reductive} \iff \text{reductive} \]
and all three notions coincide in characteristic zero. The first implication is immediate from the definitions and, in characteristic zero, the opposite implication goes back to Weyl and uses the representation theory of compact Lie groups (this is known as Weyl's unitary trick). The equivalence between reductive and geometrically reductive for smooth affine group schemes was conjectured by Mumford after Nagata proved that every geometrically reductive group is reductive \cite{nagata}; the opposite implication was proved by Haboush \cite{haboush}.

Let us state an important property of geometrically reductive group actions.

\begin{lemma}[Geometrically reductive group actions separate closed orbits, {\cite[Lemma 3.3]{newstead}}] \label{geom red lem}
Let $G$ be a geometrically reductive group acting on an affine scheme $X$. If $W_1$ and $W_2$ are disjoint $G$-invariant closed subsets of $X$, then there is an invariant function $f \in \cO(X)^G$ which separates these sets, i.e.
\[ f(W_1) = 0 \quad \mathrm{and} \quad f(W_2) = 1. \]
\end{lemma}

\subsection{Finitely generated rings of invariants}\label{sec fg inv}

Recall that for an action of an affine algebraic group $G$ on an affine scheme $X$, the associated action on the coordinate ring $\cO(X)$ is rational. 

\begin{thm}[Nagata, \cite{nagata}] \label{thm Nagata fgen}
Let $G$ be a geometrically reductive group acting rationally on a finitely generated $k$-algebra $A$. Then the $G$-invariant subalgebra $A^G$ is finitely generated.
\end{thm}

We will outline the proof of this theorem in the significantly easier case when $G$ is linearly reductive; see \cite[Theorem 3.4]{newstead} for the full proof. In this case, one can construct a \emph{Reynolds operator}, which is a projection $R : A \twoheadrightarrow A^G$ onto the $G$-invariants that satisfies $R(ab) = aR(b)$ for all $a \in A^G$ and $b \in A$. If $G$ is finite, $R$ can be viewed as averaging over the group. Using the Reynolds operator, one can show that $A^G$ is Noetherian and then prove it is finitely generated. This approach is similar to Hilbert's proof that over the complex numbers the ring of invariants for $\GL_n$ is finitely generated.

\begin{proof}[Proof of Theorem \ref{thm Nagata fgen} (for linearly reductive groups)] Since $A$ is a finitely generated $k$-algebra, it has a countable basis as a $k$-vector space; thus $A$ can be written as an increasing union of finite dimensional vector spaces. By applying Lemma \ref{fdim lemma} to these vector spaces, we can write $A$ as an increasing union of finite dimensional $G$-invariant vector spaces $W_n$ over $n \in \NN$.

Our assumption that $G$ is linearly reductive implies that the finite dimensional $G$-representation $W_n$ is completely reducible. In particular, we can write $W_n$ as a sum of $G$-representations
\[ W_n = W_n^G \oplus W_n'\]
and obtain a projection $R_n : W_n \twoheadrightarrow W_n^G$, which together induce a projection $R : A \ra A^G$.

To show that this projection is a Reynolds operator, we need to show $R(ab) = aR(b)$ for all $a \in A^G$ and $b \in A$. For this take $n$, so $a,b \in W_n$ and pick $m \geq n$ such that left multiplication $l_a : A \ra A$ restricts to a homomorphism of $G$-representations 
\[l_a : W_n \ra W_m.\]
As above, we write $W_n = W_n^G \oplus W_n'$. Since $a \in A^G$, we have $l_a(W_n^G) \subset W_m^G$ and by Schur's Lemma, the image of each irreducible representation appearing in $W_n'$ is either zero or isomorphic to that irreducible representation, thus $l_a(W_n') \subset W_m'$. If we write $b = b^G + b' \in W_n^G \oplus W_n'$, then 
\[ ab = l_a(b) = l_a(b^G) + l_a(b')=ab^G + ab' \in W_m^G \oplus W_m'.\] 
Hence, $R(ab) = ab^G = a R(b)$ as required. 

For any ideal $I \subset A^G$, we have $I \subset IA \cap A^G$, and using the Reynolds operator, one can show the opposite inclusion. Hence $I = IA \cap A^G$ and from this we deduce $A^G$ is Noetherian: any increasing chain of ideals $I_n$ in $A^G$ must stabilise, as the corresponding chain of ideals $I_nA$ stabilises due to $A$ being Noetherian.

By choosing generators for the $k$-algebra $A$, we can realise it as a quotient of a polynomial ring with linear $G$-action $\Sym^*(V) \twoheadrightarrow A$. Since any $G$-equivariant homomorphism of algebras commutes with their Reynolds operators, we obtain a surjection $\Sym^*(V)^G \twoheadrightarrow A^G$ and so to show $A^G$ is finitely generated, it suffices to show $\Sym^*(V)^G$ is finitely generated. Thus we may assume $A = \Sym^*(V)$ is a polynomial ring with linear action $G \ra \GL(V)$. Since $A^G$ is Noetherian, the ideal $A_+^G := \oplus_{n > 0} \Sym^n(V)^G$ is finitely generated and the generators of this ideal are generators of $A^G$ as a $k$-algebra. 
\end{proof}

Popov \cite{Popov} proved a converse to Nagata's theorem: for any non-reductive group $G$ there is an affine scheme $X$ such that $\cO(X)^G$ is not finitely generated.

In some simple situations, the ring of invariants for a non-reductive group is finitely generated; however, the corresponding morphism of schemes may fail to be a good quotient (see $\S$\ref{sec additive bad}).

\begin{thm}[Weitzenb\"{o}ck \cite{Weitzenbock}]\label{thm weitzenbock}
Assume that the characteristic of $k$ is zero, then any linear $\GG_a$-action on $\AA^n$ extends to $\SL_2$. In this case, the invariant ring $\cO(\AA^n)^{\GG_a}$ is finitely generated.
\end{thm}
\begin{proof}[Proof (after Seshadri \cite{Seshadri_weitzenbock})]
The first statement follows by putting the associated locally nilpotent derivation (see $\S$\ref{sec Ga LND}) in Jordan normal form, so that each length $n$ Jordan block corresponds to the standard $\SL_2$-representation $\Sym^{n-1}(k^2)$; for details, see \cite[Lemma 10.2]{GrosshansBook}.

Assuming that the linear $\GG_a$-action extends to $\SL_2$, let us prove that the ring of invariants is finitely generated. By the transfer principle (Remark \ref{rmk transfer principle}), the ring of $\GG_a$-invariants on $\AA^n$ is isomorphic to the ring of $\SL_2$-invariants on $\AA^n \times^{\GG_a} \SL_2$.

For $\GG_a$-acting on $\SL_2$ by left multiplication, the bottom row is invariant and thus the map $\SL_2 \ra  \AA^{2} \setminus \{ 0 \}$ given by sending a matrix $A \in \SL_2$ to its bottom row $(a_{21},a_{22})$ is $\GG_a$-invariant. In fact, this map is an orbit space and $\SL_2 /\GG_a \cong \AA^{2} \setminus \{ 0 \}$; one way to see this is to note that $\GG_a$ is the $\SL_2$-stabiliser of $(1,0) \in \AA^2$ and its orbit $\SL_2 \cdot (1,0) = \AA^{2} \setminus \{ 0 \}$ is isomorphic to $\SL_2/\GG_a$.  Since the $\GG_a$-action on $\AA^n$ extends to $\SL_2$, we have an $\SL_2$-equivariant isomorphism
\[ \AA^n \times^{\GG_a} \SL_2 \cong \AA^n \times \SL_2/\GG_a. \]
This is only quasi-affine (so its coordinate ring may not be finitely generated), but as $ \{ 0 \} \subset \AA^2$ has codimension 2, any regular function extends from $\AA^{2} \setminus \{ 0 \} $ to $\AA^2$ by Hartogs' lemma. Hence,
\[ \cO(\AA^n)^{\GG_a} \cong \cO(\AA^n \times^{\GG_a} \SL_2)^{\SL_2} \cong \cO(\AA^n \times \SL_2/\GG_a)^{\SL_2} \cong \cO(\AA^{n+2})^{\SL_2} \]
is finitely generated.
\end{proof}

\begin{rmk}\label{rmk Grosshans subgp}
The second part of the above proof for $H = \GG_a < G = \SL_2$ can be extended to any \emph{Grosshans subgroup}, which is a closed subgroup $H < G$ of a reductive group such that $G/H$ is quasi-affine and $\cO(G/H)=\cO(G)^H$ is finitely generated (see \cite{GrosshansBook} for a detailed treatment). Grosshans \cite{Grosshans_obs} shows $\cO(G)^H$ is finitely generated if and only if $G/H$ can be emdedded in an affine variety with complement of codimension $2$, and proves that unipotent radicals of parabolic subgroups in a reductive group are Grosshans subgroups \cite{Grosshans_inv}. For a Grosshans subgroup $H < G$, the same proof shows that \emph{if} a $H$-action an an affine scheme $X$ \emph{extends to $G$} (which is not immediate as in the case of Weitzenb\"{o}ck's Theorem), then $\cO(X)^H$ is finitely generated.
\end{rmk}

The proof shows that for a non-reductive group, even if the ring of invariants is finitely generated, taking its spectrum does not necessarily provide a categorical quotient: for $\GG_a$-acting on $\SL_2$, we have $\cO(\SL_2)^{\GG_a} = k[x_{21},x_{22}]$, but the induced map $\SL_2 \ra \AA^2$ is not surjective, so $\AA^2$ is not the categorical quotient. In fact, even worse, the image may only be a constructible subset  (see $\S$\ref{sec additive bad}). In the next subsection we will see that when $G$ is reductive, taking the spectrum of the ring of invariants does give a categorical quotient.

\subsection{Affine geometric invariant theory for reductive groups}\label{sec affine GIT}

Let $G$ be a reductive group acting on an affine scheme $X$. There is an induced action of $G$ on the coordinate ring $\cO(X)$ and the ring of invariants $\cO(X)^G$ is a finitely generated $k$-algebra by Nagata's Theorem.

\begin{defn}[Affine GIT quotient]
For an action of a reductive group $G$ on an affine scheme $X$, the \emph{affine GIT quotient} is the morphism $\varphi : X \ra X/\!/G:=\Spec \cO(X)^G$ of affine schemes associated to the inclusion $\varphi^*: \cO(X)^G \hookrightarrow \cO(X)$. 
\end{defn}

The double slash notation $X /\!/G$ used for the GIT quotient is a reminder that this quotient is not necessarily an orbit space and so it may identify some orbits. In nice cases, the GIT quotient is an orbit space and in this case we shall write $X/G$.

\begin{thm}[Mumford, {\cite[Theorem 1.1]{mumford}}]
For a reductive group $G$ acting on an affine scheme $X$, the affine GIT quotient $X \ra X/\!/G$ is a good quotient and thus categorical quotient. 
\end{thm}

We will not include the proof of this result, but we note that the proof that the affine GIT quotient is good uses several properties of reductive group actions beyond simply the finite generation of the invariant ring: for a geometrically reductive group, invariant functions can be used to separate closed orbits (see Lemma \ref{geom red lem}) and, for linearly reductive groups, the proof can be simplified by using the fact that taking invariants is exact.

The affine GIT quotient restricts to a geometric quotient on an open stable subset $X^s \subset X$.

\begin{defn} 
A point $x \in X$ is \emph{stable} if its orbit is closed in $X$ and $\dim G_x=0$ (or equivalently, $\dim G \cdot x = \dim G$). We let $X^{s}$ denote the set of stable points.
\end{defn}

For $x \in X$, we note that $x$ is stable if and only if $\sigma_x : G \ra X$ is proper. Indeed if $\sigma_x$ is proper, then its image $G \cdot x$ is closed and the fibres, being both affine and proper, must be finite. Conversely if $x$ is stable, then $\sigma_x : G \ra G \cdot x$ has finite fibres and one can show it is finite.

\begin{ex}\label{ex wts plus minus 1}
For the $\GG_m$-action on $\AA^2$ by $t \cdot(x,y )=(tx,t^{-1}y)$, we have  $\cO(\AA^2)^{\GG_m} = k[xy]$ with affine GIT quotient $\varphi : \AA^2 \ra \AA^1$ is given by $(x,y) \mapsto xy$. It is not a geometric quotient, as the three orbits consisting of the punctured axes and the origin are all identified. The stable locus is the complement of $xy = 0$, which admits a geometric quotient $\AA^1 \setminus \{ 0 \}$.

If we remove the origin, the affine line with a double origin is a geometric quotient of $\AA^2 \setminus \{ 0 \}$. In this case, we obtain a non-separated quotient as a categorical quotient of a separated scheme.
\end{ex}.

\begin{ex}\label{ex conj 2 x 2}
Consider $G=\GL_2$ acting by conjugation on the space $\Mat_{2 \times 2}$ of $2 \times 2$ matrices with $k$-coefficients. The trace and determinant (which are the coefficients of the characteristic polynomial) are invariant functions, and so 
\[ k[\text{tr}, \det] \subset \cO(M_{2 \times 2})^{\GL_2}.\]
We will soon see this is in fact an equality.

First, we describe the orbits using the theory of Jordan normal forms. As any orbit contains a matrix in Jordan normal form, there are three types of orbits:
\begin{itemize}
\item Matrices with distinct eigenvalues $\alpha,\beta$ and Jordan normal form
\[ \left( \begin{array}{cc} \alpha & 0 \\ 0 & \beta \end{array} \right). \] These are closed 2 dimensional orbits, with 2 dimensional stabiliser (diagonal matrices).
\item Matrices with repeated eigenvalue and Jordan normal form with one block
\[ \left( \begin{array}{cc} \alpha & 1 \\ 0 & \alpha \end{array} \right). \]
These orbits are also 2 dimensional but are not closed: for example
\[ \lim_{t \to 0}  \left( \begin{array}{cc} t & 0 \\ 0 & t^{-1} \end{array} \right) \left( \begin{array}{cc} \alpha & 1 \\ 0 & \alpha \end{array} \right) \left( \begin{array}{cc} t^{-1} & 0 \\ 0 & t \end{array} \right)= \left( \begin{array}{cc} \alpha & 0 \\ 0 & \alpha \end{array} \right). \]
\item Matrices with repeated eigenvalue and Jordan normal form with two blocks
\[ \left( \begin{array}{cc} \alpha & 0 \\ 0 & \alpha \end{array}\right). \]
The stabiliser of such a matrix is $\GL_2$ and its orbit is a point, which is closed.
\end{itemize}
Every orbit closure of the second type contains an orbit of the third type. 

Let us show that $ \cO(\Mat_{2 \times 2})^{\GL_2} = k[\tr,\det]$. Since any orbit closure contains a diagonal matrix, any invariant function is completely determined by its values on the diagonal matrices and is invariant under permuting the diagonal entries. Hence
\[ \cO(\Mat_{2 \times 2})^{\GL_2} \subset k[x_{11},x_{22}]^{S_2}= k[x_{11} + x_{22}, x_{11}x_{22}] = k[\tr,\det]\]
by the theory of (elementary) symmetric polynomials.

The affine GIT quotient is $ \varphi = (\mathrm{tr},\det): \Mat_{2 \times 2} \ra \Mat_{2 \times 2}/\!/\GL_2 = \AA^2$. Since scalar multiples of the identity fix every point, there are no stable points for this action; however, the restriction to the locus of matrices with distinct eigenvalues is a geometric quotient. 
\end{ex}

\begin{exer}
Show the GIT quotient of $\GL_n$ acting on $\Mat_{n \times n}$ by conjugation is $\AA^n$.
\end{exer}

Newstead constructs moduli spaces of cyclic endomorphisms of vector spaces  \cite[Chapter 2]{newstead}.

\subsection{Projective geometric invariant theory}\label{sec proj GIT}

Suppose that a reductive group $G$ acts on a projective scheme $X \subset \PP^n$ linearly (i.e.\ by a representation $G \ra \GL_{n+1}$). The homogeneous coordinate ring of $X$ is the graded ring
\[ R(X) =  k[x_0,\dots , x_n]/I_X = \bigoplus_{r \geq 0} H^0(X,\cO(r))\]
where $\cO(1)$ denotes the pullback of $\cO_{\PP^n}(1)$ to $X$. By Nagata's theorem, $R(X)^G$ is finitely generated. The inclusion $R(X)^G \hookrightarrow R(X)$ determines a rational map of projective schemes
\begin{equation}\label{proj of invariants}
 X \dasharrow \proj R(X)^G
\end{equation}
whose indeterminacy locus is the closed subscheme of $X$ defined by the homogeneous ideal $R(X)_+^G:= \oplus_{r > 0} R(X)_r^G$. The domain of definition of this map is the GIT semistable locus.

\begin{defn}
Let $G$ be a reductive group acting linearly on a projective scheme $X \subset \PP^n$. 
\begin{enumerate}[label={\roman*)}]
\item We say $ x\in X$ is \emph{semistable} if there exists a $G$-invariant homogeneous function $f \in R(X)^G_r$ for some $r > 0$ such that $f(x) \neq 0$. We write $X^{ss}$ for the open set in $X$ of semistable points; this is the domain of definition of \eqref{proj of invariants}.
\item We say $x \in X$ is \emph{stable}\footnote{Usually stability is defined by asking for $\dim G_x = 0$ and for the existence of $f \in R(X)^G_r$ for some $r > 0$ non-vanishing at $x$ such that the $G$-action on $X_f$ is closed; however, this is equivalent to the stated definition.} if its orbit is closed in $X^{ss}$ and its stabiliser is zero dimensional. We write $X^{s}$ for the open set in $X$ of stable points.
\item The restriction of the rational map \eqref{proj of invariants} to the semistable locus $X^{ss} \ra X/\!/G := \proj R(X)^G$ is called the \emph{projective GIT quotient}, which is projective over $k$.
\end{enumerate}
\end{defn}

Rather confusingly, a point is called \emph{unstable} if it is not semistable; this terminology is now standard and there is not much we can do to change it! We refer to points which are semistable but not stable, as \emph{strictly semistable}.
If there are several groups acting on $X$, we clarify which group we mean by talking about \emph{$G$-(semi)stability}.

By definition, $X^{ss}$ is open as it is the domain of definition of the rational map \eqref{proj of invariants}. To see that $X^s$ is open, we use the equivalent formulation and note it is the intersection of two opens: the set of points with zero dimensional stabiliser is open as $x \mapsto \dim G_x$ is upper semi-continuous and the union of $X_f$ for $f \in R(X)^G_+$ on which the action on $X_f$ is closed is open.

\begin{thm}[Mumford, see {\cite[Theorem 3.14]{newstead}}] 
For a reductive group $G$ acting linearly on a projective scheme $X \subset \PP^n$, the projective GIT quotient $\varphi: X^{ss} \ra X/\!/G$ is a projective and good quotient, which restricts to a quasi-projective and geometric quotient of $X^s$. 
\end{thm}

This result can be proved by gluing together affine GIT quotients: for $f \in R(X)_+^G$, the non-vanishing locus $X_f$ is affine with affine GIT quotient $X_f \ra X_f /\!/G$ and we can write $X^{ss}$ as the union of these open affines $X_f$, so $X/\!/G$ is covered by the open affines $X_f /\!/G$.

We have $\varphi(x) = \varphi(y)$ if and only if the orbit closures of $x$ and $y$ meet in $X^{ss}$. Furthermore, the preimage of any point in $X/\!/G$ contains a unique closed orbit (of minimal dimension in this preimage), whose stabiliser is reductive (see Remark \ref{rmk conseq good}).

\begin{rmk}
It is important to note that the semistable set and the GIT quotient both depend on the $G$-equivariant embedding $X \hookrightarrow \PP^n$, as the homogeneous coordinate ring depends on this embedding (or equivalently on the line bundle $\cO(1)$ pulled back from $\PP^n$). 

Alternatively, rather than fixing a linear $G$-equivariant projective embedding of $X$, one can instead fix an ample \emph{$G$-equivariant} line bundle $\cL$ on $X$, which is often called an ample \emph{$G$-linearisation}: $\cL = (L,\Phi)$ is an ample invertible sheaf $L$ on $X$ together with a $G$-equivariant structure given by an isomorphism $\Phi : \sigma^*L \ra \pi_2^* L$, where $\sigma, \pi_2 : G \times X \ra X$ denote the action and second projection, which satisfies a cocycle condition $\pi_{23}^* \Phi \circ (\Id_G \times \sigma)^* \Phi = (m \times \Id_G)^* \Phi$ on $G \times G \times X$. In terms of the associated geometric line bundle, which by abuse of notation we shall also call $L$, this is equivalent to a $G$-action on $L$ commuting with the projection $L \ra X$ such that the action on the fibres $L_{g \cdot x} \ra L_{x}$ is linear.

Given an ample \emph{$G$-equivariant} line bundle $\cL$ on $X$, we obtain a graded ring with a $G$-action
\[ R(X,\cL) = \bigoplus_{r \geq 0} H^0(X,L^{\oplus r}) \]
such that the inclusion of invariants induces a rational map whose domain of definition is the semistable set and whose codomain is the GIT quotient (both with respect to $\cL$)
\[ X^{ss}(\cL) \ra X/\!/_{\!\cL} G := \proj  R(X,\cL)^G. \]
Since replacing $\cL$ with a positive power just has the effect of changing the grading on this ring\footnote{By construction, the projective GIT quotient comes with a line bundle and this regrading does not change the GIT quotient but does change this line bundle.}, we can assume $\cL$ is very ample and then we obtain a linear $G$-equivariant embedding $X \hookrightarrow \PP(V)$ where $V := H^0(X,L)^*$, which recovers the above setting of a linear action.

The effect of changing $\cL$ is called \emph{variation of GIT} and can be described in terms of certain birational transformations known as VGIT flips \cite{dh,thaddeus}. Furthermore, the space of $G$-linearisations admits a wall and chamber decomposition describing how semistability varies: in chambers, semistability coincides with stability, but semistability changes on crossing a wall.
\end{rmk} 

\subsection{General GIT quotients}\label{sec general GIT}

More generally, given a scheme $X$ with a $G$-linearisation $\cL$, Mumford defines a GIT quotient using invariant sections of positive powers of $L$ whose non-vanishing locus is affine (so that one can take affine GIT quotients and glue them). This produces a good quotient of a \lq semistable locus' (\cite[Definition 1.7]{mumford}), which in this situation is defined to be the set of points $x \in X$ such that there exists $\sigma \in H^0(X,L^{\oplus r})^G$ for $r >0$ with $\sigma(x) \neq 0$ and such that $X_\sigma$ is affine\footnote{If $X$ is projective and $\cL$ is ample, then this non-vanishing locus is always affine.}. The semistable set and quotient obtained in this way are both quasi-projective (see \cite[Theorem 3.21]{newstead}).

Let us remark that in this survey we have assumed that we are working over an algebraically closed field $k$. The assumption that $k$ is algebraically closed can be dropped, but one has to be careful about rationality questions and work with geometric points for certain statements (for example, the Hilbert--Mumford criterion). Moreover, Seshadri \cite{Seshadri_GIT} extended GIT to work relative to a base scheme $S$ with mild assumptions on $S$. 

\subsection{Affine GIT linearised by a character}\label{sec affine GIT twisted char}

As a special case of $\S$\ref{sec general GIT}, consider a linear action of $G$ on an affine scheme $X \subset \AA^n$; then the structure sheaf $\cO_X$ is naturally equipped with a $G$-equivariant structure, where if we view this as a geometric line bundle $X \times \AA^1$, the $G$-action on $\AA^1$ is trivial. In this case, the GIT quotient with respect to this ample $G$-linearisation $\cO_X$ is just the affine GIT quotient, as 
\[ R(X,\cL) = \cO(X)[z] \]
with trivial $G$-action on $z$ and this ring is graded by the degree of $z$, thus \[ X /\!/_{\! \cO_X} G := \proj R(X,\cL)^G = \proj \cO(X)^G[z] = \spec \cO(X)^G = X/\!/ G. \]

This linearisation can be modified by using a character $\rho : G \ra \GG_m$ to obtain a linearisation $\cO_{\rho}$ which is given by $G$ acting linearly on the geometric line bundle $X \times \AA^1$ by the given action on $X$ and acting via multiplication with $\rho$ on $\AA^1$. The outcome of applying GIT in this situation of twisting the linearisation by a character was described by King \cite{king} and results in an open subset $X^{\rho-ss}$ of $\rho$-semistable points and a GIT quotient 
\[ X^{\rho-ss} \ra X/\!/_{\!\!\rho} G:= \proj \bigoplus_{r \geq 0} H^0(X,(\cO_\rho)^{\otimes r})^G. \]
In this case, the $G$-invariant sections of $(\cO_{\rho})^{\otimes r} \cong \cO_{\rho^r}$ are $f \in \cO(X)$ with $f(g \cdot x) = \rho^r(g) f(x)$ for all $g\in G$ and $x \in X$, which we refer to as \emph{$\rho$-semi-invariant functions of weight $r$}. By definition, $x$ is $\rho$-semistable if there exists a $\rho$-semi-invariant function of weight $r >0$ which is non-vanishing at $x$. Furthermore, $X/\!/_{\!\!\rho} G$ is projective over the spectrum of the $0$th-graded piece which is just the affine GIT quotient $X/\!/G = \spec \cO(X)^G$.

\begin{ex}\label{ex Pn as twisted affine quotient}
For $\GG_m$ acting on $\AA^n$ by scalar multiplication linearised by $\cO_{\rho}$ for $\rho : \GG_m \ra \GG_m$ given by $t \ra t$, the coordinate functions are $\rho$-semi-invariant functions of weight $1$ and these generate the ring of invariants. Consequently, we obtain the GIT quotient
\[ (\AA^n)^{\rho -ss} = \AA^n \setminus \{ 0 \} \ra \AA^n /\!/_{\!\!\rho} \GG_m = \proj k[x_1,\dots,x_n] = \PP^{n-1}.\]
\end{ex}

\section{Semistability and instability in reductive GIT}

Since the reductive GIT quotient only provides a quotient of an open semistable locus, this naturally leads to two questions: can we describe the semistable points and what can we say about unstable (i.e.\ not semistable) points? For actions on projective (over affine) schemes, the first question is tackled by the Hilbert--Mumford criterion for semistability described in $\S$\ref{sec HM}. In moduli problems with a natural notion of subobjects, the Hilbert--Mumford criterion often gives a clean moduli-theoretic interpretation of GIT semistability. We state some application of reductive GIT to moduli in $\S$\ref{sec app GIT moduli}. We then turn to the second question in $\S$\ref{sec instability} and describe how work of Kempf \cite{kempf}, Hesselink \cite{hesselink}, Kirwan \cite{kirwan} and Ness \cite{ness} gives a stratification of the unstable locus, with a largely combinatorial flavour; we survey some applications of these stratifications and discuss the question of construction quotients of unstable strata, where naturally non-reductive groups (namely, parabolic subgroups, representing an instability flag) appear.

\subsection{Semistability and the Hilbert--Mumford criterion}\label{sec HM}

By definition, semistability in reductive GIT is given in terms of the existence of a non-vanishing invariant section. From this definition, it is extremely challenging to determine semistability, as it is essentially equivalent to computing invariant rings, which is a notoriously challenging problem. Fortunately, in certain situations (projective GIT, affine GIT linearised by a character or more generally a projective over affine set-up), the Hilbert--Mumford criterion reduces semistability to checking semistability for $\GG_m$-actions, which in turn can be combinatorially described using the weights of the action. More precisely, a $G$-semistable point is semistable for any subgroup, and thus in particular, for any $\GG_m$ contained in $G$; the Hilbert-Mumford criterion gives a converse to this statement.

For simplicity, throughout this section, we assume we have a linear representation $G \ra \GL(V)$ of a reductive group $G$ and consider the associated linear action on $X = \PP(V)$. We will describe the semistable points in this setting. For a closed subscheme $Y \subset X$ with a linear $G$-action, we have $Y^{ss} = Y \times_{X} X^{ss}$ and so it suffices to understand semistability on the ambient projective space. For an ample $G$-linearisation $\cL$ on $X$, using a power of $\cL$ puts us in this linear setting.

We will see several different versions of the Hilbert--Mumford criterion, which make it possible to determine semistability in practice. The first, and weakest, version is a topological criterion.

\begin{prop}[Topological Hilbert--Mumford criterion, {\cite[Proposition 2.2]{mumford}}] \label{prop HM top}
For a linear action of a reductive group $G$ on $\PP(V)$, the following statements hold for $x=[v] \in \PP(V)$.
\begin{enumerate}[label=\emph{\roman*)}]
\item $x$ is semistable if and only if $0 \notin \overline{G \cdot v}$;
\item $x$ is stable if and only if $\dim G_{v} = 0$ and $G \cdot v$ is closed in $V$.
\end{enumerate}
\end{prop}
\begin{proof}
We will just give the proof of the first statement. By definition $x=[v]$ is semistable if and only if there is a $G$-invariant homogeneous polynomial $f \in R(X)^G$ which is non-zero at $x$. Since $f$ is $G$-invariant it is constant on orbit closures, and so $f$ separates the closed schemes $\overline{G \cdot v}$ and $0$, which shows these closed subschemes are disjoint. Conversely, if the closed $G$-invariant schemes $\overline{G \cdot v}$ and $0$ in $V$ are disjoint, then as $G$ is geometrically reductive, there exists a $G$-invariant polynomial $f \in \cO(V)^G$ separating these subsets
\[ f(\overline{G \cdot v})=1 \quad \text{and} \quad f(0) = 0\]
by Lemma \ref{geom red lem}. By considering the decomposition of $f = \sum_i f_i$ into ($G$-invariant) homogeneous pieces, we see there is a $G$-invariant homogeneous piece $f_i$ which is non-vanishing at $x$.
\end{proof}

\begin{defn}
For a linear action of a torus $T = \GG_m^n$ on $\PP(V)$, consider the associated weight decomposition $V = \oplus_{\chi \in X^*(T)} V_\chi$. We refer to the support of this decomposition as the $T$-weights on $\PP(V)$. For $x = [v] \in \PP(V)$, we write $v = \sum v_\chi$ and define the \emph{$T$-weight set} of this point to be
\[ \wt_T(x) = \wt_T(v) = \{ \chi : v_\chi \neq 0 \} \subset X^*(T) \cong \ZZ^n. \]
\end{defn}

For a $\GG_m$-action on a separated scheme, we will often use the following notation.

\begin{notn} 
If a morphism $f: \GG_m \ra S$, with $S$ separated, extends to $\tilde{f} :\AA^1 \ra S$, then this extension is unique and we write $\lim_{t \ra 0} f(t):= \tilde{f}(0)$. Similarly if $f$ extends to $\PP^1$, we write $\lim_{t \ra \infty} f(t):=\tilde{f}(\infty)$.
\end{notn}

We can now give a combinatorial description of (semi)stability for a $\GG_m$-action in terms of whether or not the origin lies in (the interior of) the convex hull of $\GG_m$-weights. 

\begin{prop}[Hilbert--Mumford for $\GG_m$-actions]\label{prop HM Gm}
For a linear action of $\GG_m$ on $\PP(V)$ and $x \in \PP(V)$, the following statements hold:
\begin{enumerate}[label=\emph{\roman*)}]
\item $x$ is $\GG_m$-semistable if and only if $ 0 \in \conv( \wt_{\GG_m}(x))$.
\item $x$ is $\GG_m$-stable if and only if $0 \in \Int (\conv( \wt_{\GG_m}(x)))$.
\end{enumerate}
\end{prop}
\begin{proof}
We again just prove the statement for semistability. By the topological Hilbert--Mumford criterion, we have that $x = [v] \in \PP(V)$ is $\GG_m$-semistable if and only if $0 \notin \overline{\GG_m \cdot v}$. Any point in the boundary of this orbit closure is either
\[ \lim_{t \ra 0} t \cdot v \quad \text{or} \quad \lim_{t \ra \infty} t \cdot v = \lim_{t \ra 0} t^{-1} \cdot v.\]
Moreover, we have $\lim_{t \ra 0} t \cdot v = 0$ if and only if $\wt_{\GG_m}(v) \subset \ZZ_{> 0}$ (and similarly $\lim_{t \ra \infty} t \cdot v = 0$ if and only if $\wt_{\GG_m}(v) \subset \ZZ_{< 0}$). Hence $x = [v] \in \PP(V)$ is $\GG_m$-semistable if and only if there exists $r_0 \leq 0$ and $r_\infty \geq \infty$ in $\wt_{\GG_m}(v)$, or equivalently $ 0 \in \conv (\wt_{\GG_m}(x))$.
\end{proof}

\begin{ex}
The linear action of $\GG_m$ on $X = \PP^{n}$ by \[t \cdot [x_0 : x_1 : \cdots : x_n] = [t^{-1} x_0 : t x_1 : \cdots : tx_n] \]
has weights $\pm 1$. Hence, for a point $x$ to be (semi)stable it needs both these weights, which means its first coordinate $x_0$ must be non-zero and at least one of the other coordinates $x_i$ for $i > 0$ must be non-zero. One can also see this by directly proving that 
\[ R(\PP^n,\cO(1))^{\GG_m} = k[ x_0x_1, \dots, x_0x_n ]. \]
In particular, $X^{ss} \cong \AA^{n} \setminus \{ 0 \}$ and $\PP^n/\!/\GG_m = \PP^{n-1}$ is a geometric $\GG_m$-quotient.
\end{ex}

The Hilbert--Mumford criterion will ultimately be a numerical criterion that phrases semistability in terms of the weights of \emph{1-parameter subgroups} (1-PS), which are non-trivial group homomorphisms $\lambda : \GG_m \ra G$. 

\begin{defn}[Hilbert--Mumford weight]
For a linear action of a reductive group $G$ on $\PP(V)$, we define the Hilbert-Mumford weight of $x= [v]$ at a 1-parameter subgroup $\lambda : \GG_m \ra G$ to be 
\[\mu(x,\lambda) := -\mathrm{min} \wt_{\lambda(\GG_m)}(x).\]
\end{defn}

Let us note some useful properties of the Hilbert--Mumford weight.

\begin{exer}\label{ex HM wt}
Show that the Hilbert--Mumford weight of $x=[v]$ has the following properties.
 \begin{enumerate} \renewcommand{\labelenumi}{(\arabic{enumi})}
 \item $\mu(x,\lambda)$ is the unique integer $\mu$ such that $\lim_{t \to 0} t^{\mu} \lambda(t) \cdot v$ exists and is non-zero.
  \item $\mu(x,\lambda) = \mu(x_0, \lambda)$ where $x_0 = \lim_{t \to 0} \lambda(t) \cdot x$ (and this limit exists as $X$ is projective).
  \item $\mu(x,\lambda) \leq 0 \iff  \lim_{t \to 0} \lambda(t) \cdot v$ exists, with equality if and only if $ \lim_{t \to 0} \lambda(t) \cdot v \neq 0$.
   \item $\mu(g \cdot x,g\lambda g^{-1})=\mu(x,\lambda)$ for all $g \in G$.
 \item $\mu(x,\lambda^n) = n \mu(x,\lambda)$ for a positive integer $n$.
\end{enumerate}
\end{exer}

For a linear $\GG_m$-action on $\PP(V)$, we see that for the 1-PS given by $\lambda(t) = t$, we have 
\[ \mu(x,\lambda) \geq 0 \iff  \lim_{t \to 0} t \cdot v \neq 0\]
and 
\[ \mu(x,\lambda^{-1}) \geq 0 \iff  \lim_{t \to \infty} t \cdot v \neq 0\]
Hence $x$ is semistable if $ \mu(x,-) \geq 0$ for $\lambda$ and $\lambda^{-1}$. Furthermore, $x$ is stable if and only if this inequality is strict for both 1-PSs. This is precisely the numerical version of the Hilbert--Mumford criterion that we now can state.

\begin{thm}[Hilbert--Mumford criterion, {\cite[Theorem 2.1]{mumford}}]\label{thm HM}
For a reductive group $G$ acting linearly on a projective scheme $X \subset \PP^n$, the following statements hold for $x \in X$.
\begin{enumerate}[label=\emph{\roman*)}]
\item $x$ is semistable if and only if $\mu(x,\lambda) \geq 0$ for all 1-PS $\lambda : \GG_m \ra G$,
\item $x$ is stable if and only if $\mu(x,\lambda) > 0$ for all 1-PS $\lambda : \GG_m \ra G$.
\end{enumerate}
\end{thm}

Note that it suffices to check these inequalities for \emph{primitive} 1-PSs (i.e.\ 1-PSs which are not positive powers of another 1-PS); see Exercise \ref{ex HM wt}.

The full proof of the Hilbert--Mumford criterion is beyond the scope of this survey, but  following the topological version (Proposition \ref{prop HM top}), it suffices to show that the reductive group $G$ has enough 1-PSs to detect if the origin is contained in the closure of orbits of linear actions $G \ra \GL(V)$, which is precisely the following result (see \cite[p53]{mumford} and \cite[Theorem 1.4]{kempf}), whose proof involves the Cartan-Iwahori decomposition for the reductive group $G$.

\begin{thm}[Fundamental Theorem of GIT]
Let $G$ be a reductive group acting on an affine space $V$. If $v \in V$ and $0 \in \overline{G \cdot v}$, then there is a 1-PS $\lambda$ of $G$ such that $\lim_{t \to 0} \lambda(t) \cdot v = 0$.
\end{thm}

\begin{rmk}[Hilbert--Mumford weight for a linearised action] \label{rmk HM weight linearised}
In the case of a linear $G$-action on a projective scheme $X \subset \PP^n$, the Hilbert--Mumford weight for $x \in X$ defined above depends on the choice of $G$-representation $G \ra \GL_{n+1}$ (as the weights depend on this representation). 

In general, for a $G$-linearisation $\cL = (L,\Phi)$ on a projective $G$-scheme $X$, we consider the $\lambda(\GG_m)$-fixed point $x_0 = \lim_{t \ra 0} t \cdot x$. The linearisation $\Phi$ induces a $\GG_m$-representation on the fibre of $L$ over $x_0$
\[  L_{x_0} = L_{\lambda(t) \cdot x_0} \stackrel{\cdot \lambda(t)^{-1}}{\longrightarrow} L_{x_0}\]
of weight $r$ (that is $\lambda(t)^{-1}$ acts on this fibre by $t \mapsto t^r$). Then the Hilbert--Mumford weight (with respect to $\cL$) is defined to be minus the weight on this fibre
\[ \mu^{\cL}(x,\lambda) = - r. \]
In this linearised situation, the Hilbert--Mumford criterion says $x \in X$ is semistable (with respect to $\cL$) if and only if $\mu^{\cL}(x,\lambda) \geq 0$ for all 1-PS $\lambda : \GG_m \ra G$.

If $X \subset \PP^n$ and $\cL = \cO(1)$ is the pullback of $\cO_{\PP^n}(1)$, then these two definitions coincide: we have $\mu^{\cO(1)}(x,\lambda) = \mu(x,\lambda)$ by \cite[Proposition 2.3]{mumford}.
\end{rmk}

As any 1-PS can be conjugated to lie in a fixed maximal torus $T < G$, one can phrase $G$-semistability of a point in terms of $T$-semistability of all $G$-translates of that point by Exercise \ref{ex HM wt} above. Then $T$-semistability can be stated combinatorially using the torus weights analogous to Proposition \ref{prop HM Gm} above. This gives a combinatorial Hilbert--Mumford criterion.

\begin{prop}[Torus weights version of Hilbert--Mumford criterion, {\cite[$\S$9.4]{dolgachev}}]
For a reductive group $G$ acting on a projective scheme $X \subset \PP^n$ linearly, fix a maximal torus $T < G$. For $x \in X$, the following statements hold.
\begin{enumerate}[label=\emph{\roman*)}]
\item $x$ is $G$-(semi)stable if and only if $g \cdot x$ is $T$-(semi)stable for all $g \in G$.
\item $x$ is $T$-semistable if and only  if $ 0 \in \conv( \wt_{T}(x))$.
\item $x$ is $T$-stable if and only if $0 \in \Int (\conv( \wt_{T}(x)))$.
\end{enumerate}
By the first statement, the $G$-semistable set is the $G$-sweep of the $T$-semistable set:
\[ X^{G-ss} = \bigcap_{g \in G} g \cdot X^{T-ss}. \]
\end{prop}

\begin{exer}[Semistability for binary forms]\label{exer ss binary forms}
Consider the action of $\SL_2$ on the space of degree $d$ binary forms $\PP^d = \PP(k[x,y]_d)$. For $p_F \in \PP^d$ corresponding to $F(x,y) \in k[x,y]_d$, show
\begin{enumerate}[label=\emph{\roman*)}]
\item $F$ is semistable if and only if all roots of $F$ have multiplicity less than or equal to $d/2$;
\item $F$ is stable if and only if all roots of $F$ have multiplicity strictly less than $d/2$.
\end{enumerate}
\end{exer}

\begin{rmk}[Hilbert--Mumford criterion for action on affine scheme twisted by a character]\label{rmk HM affine char} For a linear action of a reductive group $G$ on an affine scheme $X \subset \AA^n$ linearised via $\cO_\rho$ for a character $\rho : G \ra \GG_m$ (see $\S$\ref{sec affine GIT twisted char}), King proved a topological Hilbert--Mumford criterion \cite[Lemma 2.2]{king}, by using the total space of the dual linearisation to replace the affine cone, and obtained the following numerical Hilbert--Mumford criterion \cite[Proposition 2.5]{king}: 
\begin{enumerate}
\item $x$ is $\rho$-semistable if and only if $\langle \rho, \lambda \rangle \geq 0$ for all 1-PS $\lambda : \GG_m \ra G$ such that $\lim_{t \ra 0} \lambda(t) \cdot x$ exists,
\item $x$ is $\rho$-stable if and only if $\langle \rho, \lambda \rangle > 0$ for all 1-PS $\lambda : \GG_m \ra G$ such that $\lim_{t \ra 0} \lambda(t) \cdot x$ exists,
\end{enumerate}
where $\langle \rho, \lambda \rangle = r$ if $\rho \circ \lambda(t) = t^r$, i.e.\ this is the natural pairing between characters and cocharacters. Using the abstract definition of the Hilbert--Mumford weight in terms of the weight of the action on the fibre over the limit point of the $\GG_m$-action (see Remark \ref{rmk HM weight linearised}), we see that if $x_0 = \lim_{t \ra 0} \lambda(t) \cdot x$ exists, then $ \mu^{\cO_\rho}(x,\lambda) = \langle \rho, \lambda \rangle.$
\end{rmk}

\begin{exer}\label{exer grassmannain}
Using Remark \ref{rmk HM affine char}, show that for $\GL_2$ acting on $\Mat_{2 \times n}$ by left multiplication the (semi)stable locus for the character $\rho = \det$ is the matrices of maximal rank (namely rank $2$), and so the GIT quotient is the Grassmannian $\Gr(2,n)$.
\end{exer}

We note that there is a Hilbert--Mumford criterion in the more general setting of a projective over affine variety with an action of a linearly reductive group \cite{ghh}.

\subsection{A brief survey of applications of reductive GIT to moduli}\label{sec app GIT moduli}

The notion of moduli functor is heavily influenced by Grothendieck's approach to algebraic geometry. Furthermore, Grothendieck proved that the Hilbert and Quot functors are representable by projective schemes; these are fine moduli spaces and provide parameter spaces in the GIT constructions of moduli of smooth projective curves and moduli of vector bundles on curves.

The first truly interesting application of GIT was Mumford's construction of moduli spaces of curves \cite[Chapter 5]{mumford}. Mumford  constructed a coarse moduli space $M_g$ for smooth projective curves of genus $g \geq 2$ by using a power of the canonical bundle to give a projective embedding $C \hookrightarrow \PP^{N}$ and constructing $M_g$ as a quotient of a suitable Chow variety parametrising pluricanonical curves. Gieseker  \cite{Giesker_curves} provided an alternative GIT construction of $M_g$ and its Deligne--Mumford compactification $\overline{M_g}$ via stable curves as a quotient of the $\PGL_{N+1}$-action on a suitable Hilbert scheme with a linearisation given by embedding in a Grassmannian associated to a sufficiently large choice of $m$. Although there is a direct proof that smooth curves are (asymptotically) GIT stable, the proof that stable curves are (asymptotically) GIT stable is indirect (see \cite[$\S$3.1]{Morrison}). Gieseker's Hilbert scheme construction is now the prevalent perspective, which has been generalised to give GIT constructions of moduli spaces of pointed stable curves and stable maps and their (birational) geometry is studied using VGIT (see \cite{Laza, Morrison}).

The other influential and successful application of GIT was the construction of moduli spaces of vector bundles (of fixed rank and degree) on a fixed smooth projective curve $C$. One of the first ideas to construct vector bundle moduli spaces over $k = \CC$ was to use unitary representations of the fundamental group $\pi_1(C)$ and led to the Narasimhan--Seshadri Theorem \cite{ns} relating irreducible representations with \emph{stable} vector bundles considered by Mumford \cite{Mumford_vb}, where Mumford's notion of stability came from the Hilbert--Mumford criterion in GIT and involves verifying an inequality of slopes (the ratio of the degree and the rank) for all subbundles. For moduli problems with a natural notion of subobjects, the study of 1-PSs in GIT often corresponds to filtrations by subobjects and stability can be phrased as an inequality for all subobjects. The GIT construction of moduli spaces of (semi)stable vector bundles was given by Seshadri \cite{seshadri} (see \cite[Chapter 5]{newstead}), and was later generalised by Simpson \cite{simpson} to construct moduli spaces of sheaves (and Higgs sheaves) on higher dimensional schemes as GIT quotients of Quot schemes. Quot schemes appear as semistable vector bundles can be parametrised as quotients of a fixed vector bundle as mentioned in Example \ref{ex moduli gp actions}\eqref{ex moduli 4}. This construction has been generalised to construct various bundle moduli spaces \cite{schmittbook}.

Mumford also applied GIT to construct moduli spaces of projective hypersurfaces $X \subset \PP^n$ of degree $d$ by taking a quotient of $\PGL_{n+1}$ acting on $\PP(k[x_1,\dots,x_n]_d)$ as in Example \ref{ex moduli gp actions} (3). He showed smooth hypersurfaces are GIT stable if $n \geq 2$ and $d \geq 3$ (see \cite[Chapter 4.2]{mumford}).

King \cite{king} developed GIT for a linear action on an affine space with respect to a character (see $\S$\ref{sec affine GIT twisted char}) to construct reasonable moduli spaces of \emph{semistable} representations of a quiver, where semistability depends on a stability parameter; the Hilbert--Mumford criterion gives a moduli-theoretic interpretation of semistability as an inequality holding for all subrepresentations.

\subsection{Instability}\label{sec instability}

In this section, we continue to suppose that we have a reductive group $G$ acting on a projective scheme $X \subset \PP^n$ linearly. Since the GIT quotient provides a categorical quotient of the semistable locus $X^{ss}$, it is natural to ask what can be said about the unstable points (i.e.\ not semistable points)
\[ X^{us} := X \setminus X^{ss}.\]
By the Hilbert--Mumford criterion, if a point is unstable, then it has a negative Hilbert--Mumford weight for some 1-PS. Starting from this observation, Kempf \cite{kempf} associated to an unstable orbit a conjugacy class of 1-PSs which are \lq most responsible' for its instability, in the sense that they minimise a \lq normalised Hilbert--Mumford weight'. Hesselink then used Kempf's work to statify the unstable locus \cite{hesselink}. This stratification was described more explicitly and, when $k = \CC$, compared with a Morse stratification associated to the norm square of the moment map for the action of a maximal compact subgroup by Kirwan \cite{kirwan} and Ness \cite{ness}.

Let us start by describing how to fix a conjugation invariant norm on 1-PSs of $G$.

\begin{defn}
A conjugation invariant norm on 1-PSs of a reductive group $G$ is given by fixing a maximal torus $T < G$ and a Weyl-invariant integral-valued bilinear form on the 1-PSs $X_*(T)$ of $T$ with associated norm $|| - ||$. For any 1-PS $\lambda : \GG_m \ra G$, there exists $g \in G$ such that $g \lambda g^{-1} \in X_*(T)$ and we define
\[ || \lambda ||:=|| g \lambda g^{-1} ||,\]
which is independent of the choice of $g$ due to the Weyl invariance.
\end{defn}

Over the complex numbers, such a norm can be constructed by fixing a Weyl invariant inner product on the Lie algebra $\mathfrak{t}$ of $T$, which gives an identification $\mathfrak{t} \cong \mathfrak{t}^*$.

\begin{ex} 
If $G = \GL_n$ and $T$ is the diagonal maximal torus, then the Euclidean norm on $\RR^n \cong X_*(T)_{\RR}$ is invariant under the Weyl group $S_n$.
\end{ex}

Using the norm $|| - ||$, we define a normalised Hilbert--Mumford weight and state Kempf's notion \cite{kempf} of an \emph{adapted} 1-PS for an unstable point.

\begin{defn}[Normalised Hilbert--Mumford weight and adapted 1-PS]
For a reductive group $G$ acting on a projective scheme $X \subset \PP^n$ linearly and a fixed conjugation invariant norm on 1-PSs of $G$, we define the \emph{normalised Hilbert--Mumford weight} of $x \in X$ at a 1-PS $\lambda$ to be $\mu(x,\lambda) / || \lambda ||$ and we define the \emph{minimum normalised Hilbert--Mumford weight of $x$} to be
\[ M(x) := \inf_{\lambda \in X_*(G)}  \frac{\mu(x,\lambda)}{|| \lambda ||}. \]
If $x$ is unstable, a primitive 1-PS is said to be \emph{adapted} to $x$ if it acheives this minimum and we write $\Lambda_x$ for the set of primitive 1-PSs adapted to $x$.
\end{defn}

Let us collect Kempf's results on adapted 1-PSs in the following theorem.

\begin{thm}[Kempf, \cite{kempf}]
Let $G$ be a reductive group acting on a projective scheme $X \subset \PP^n$ linearly and fix a conjugation invariant norm on 1-PSs of $G$. Then for an unstable point $x \in X$, we have $\Lambda_x \neq \emptyset$ and there is a parabolic subgroup $P_x < G$ with the following properties.
\begin{enumerate}[label=\emph{\roman*)}]
\item For any $\lambda \in \Lambda_x$, we have $P_x = P_\lambda$ (see Definition \ref{def para for 1-PS} below).
\item Any two 1-PSs in $\Lambda_x$ are conjugate by an element of $P_x$.
\item If $T < G$ is a maximal torus with $T < P_x$, then $\Lambda_x \cap X_*(T)$ is a single Weyl orbit.
\item We have $g \Lambda_x g^{-1} = \Lambda_{g \cdot x}$ for all $g \in G$.
\item If $\lambda \in \Lambda_x$, then $\lambda \in \Lambda_{x_0}$ and also $M(x) = M(x_0)$ where $x_0 := \lim_{t \ra 0} \lambda(t) \cdot x$.
\end{enumerate}
\end{thm}

Although we will not give the details on the proof of this theorem, the existence of an adapted 1-PS boils down to the fact that one can work in a maximal torus (by translating using the $G$-action) and then the normalised Hilbert--Mumford weight for 1-PS in a given maximal torus can be determined from subsets of the torus weights of the action, which is a finite set (see also Remark \ref{rmk inst strat finite}). The last two properties follow from Exercise \ref{ex HM wt}.

\begin{defn}[Parabolic and Levi group associated to a 1-PS]\label{def para for 1-PS}
For a 1-PS $\lambda: \GG_m \ra G$ of a reductive group $G$, we define
\[ P_\lambda := \left\{ g \in G : \lim_{t \ra 0} \lambda(t)g\lambda(t)^{-1} \text{ exists in } G \right\} \stackrel{q_\lambda}{\twoheadrightarrow} L_\lambda :=\left\{ \lim_{t \ra 0} \lambda(t)g\lambda(t)^{-1} : g \in P_\lambda \right\};  \] 
then $P_\lambda= U_\lambda \rtimes L_\lambda$ is a parabolic subgroup with Levi subgroup $L_\lambda$ and unipotent radical $U_\lambda$ and $q_\lambda: P_\lambda \ra L_\lambda$ is a retraction onto the Levi.
\end{defn}

Hesselink \cite{hesselink} stratified the unstable locus by pairs $\beta = ([\lambda],m)$ consisting of a conjugacy class of an adapted 1-PS $[\lambda]$ and a minimum normalised Hilbert--Mumford weight $m$. Before, we give the concrete construction of this stratification, we provide a summary of its properties and give an overview of various applications. As is customary, we include the semistable set as the lowest stratum in this stratification.

\begin{thm}[Kempf \cite{kempf}, Hesselink \cite{hesselink}, Kirwan \cite{kirwan}, Ness \cite{ness}]
For a reductive group $G$ acting linearly on a projective scheme $X \subset \PP(V)$ and a conjugation invariant norm on 1-PSs of $G$, there is a finite \emph{instability stratification}
\begin{equation} \label{instab strat}
X = \bigsqcup_{\beta\in \cB} S_\beta
\end{equation}
into locally closed subschemes with a partial ordered index set $\cB$ with the following properties.
\begin{enumerate}[label=\emph{\roman*)}]
\item The lowest stratum is indexed by $\beta = 0$ and we have $S_0 = X^{ss}$.
\item The closure of any stratum is contained in the union of higher strata: $\overline{S_\beta} \subset \bigsqcup_{\gamma \geq \beta} S_\gamma$.
\item $\cB$ is determined combinatorially from the weights on $V$ of a maximal torus $T < G$.
\item The strata $S_\beta$ can be determined from simpler limit sets, which are GIT semistable loci for smaller reductive group actions with a twisted linearisation.
\end{enumerate}
\end{thm}

We will soon make the last two statements more precise. First, let us state some applications of these instability (or Hesselink--Kempf--Kirwan--Ness) stratifications.

\begin{rmk}
When $X$ is smooth, these stratifications have been used in the following ways:
\begin{enumerate}
\item By Kirwan \cite{kirwan}, over $k= \CC$ compute the $G$-equivariant rational Betti numbers of $X^{ss}$ (which coincides with the rational Betti numbers of $X/\!/G$ when $X^{s} =X^{ss}$) by showing the Gysin long exact sequences in equivariant cohomology for this stratification split. Hence, there is a surjection known as the \emph{Kirwan map}
\[ H^*_G(X,\QQ) \twoheadrightarrow H^*_G(X^{ss},\QQ) \]
with explicit kernel.
\item By Dolgachev--Hu \cite{dh} and Thaddeus \cite{thaddeus}, to describe the birational transformations between GIT quotients given by varying the linearisation.
\item By Halpern-Leistner \cite{HL} and Ballard--Favero--Katzarkov \cite{BFK}, to construct semi-orthogonal decompositions in the derived category of a (stacky) GIT quotient.
\item By Halpern-Leistner \cite{HL_Theta}, as inspiration to formulate an abstract notion of a $\Theta$-stratification on a stack. 
\end{enumerate}
\end{rmk}

\begin{rmk}\label{rmk sympl GIT}
Over $k = \CC$, there is a close relationship between GIT quotients and symplectic reductions, which are quotients in symplectic geometry. A smooth projective variety $X \subset \PP^n_\CC$ is K\"{a}hler and thus has a symplectic form (inherited from the Fubini-Study form on $\PP^n_\CC$). For a representation $G \ra \GL_{n+1}$, there is a maximal compact subgroup $K < G $ which acts by unitary transformations, and $K$ preserves the K\"{a}hler form. Moreover, there is a moment map $\mu : X \ra \mathfrak{k}^*$ to the co-Lie algebra of $K$ such that the GIT quotient is homeomorphic to the symplectic reduction (the quotient of the zero level set of the moment map by $K$, see \cite{mw}):
\[ X/\!/G \simeq \mu^{-1}(0)/K \]
via the Kempf--Ness Theorem \cite{kempf_ness} (see also \cite{thomas}). More precisely, $x \in X$ is semistable if and only if its $G$-orbit closure meets $\mu^{-1}(0)$ (and if this intersection is non-empty, it consists of a unique $K$-orbit). Furthermore, by work of Kirwan \cite{kirwan} and Ness \cite{ness}, the GIT instability stratification with respect to a conjugation invariant norm coincides with the Morse stratification associated to the norm square of the moment map $||\mu||^2 : X \ra \RR$ (see also \cite[Chapter 8]{mumford}).
\end{rmk}

We now turn to the (first set-theoretic) construction of the unstable strata in the situation of a reductive group $G$ acting linearly on a projective scheme $X \subset \PP^n$ with a fixed conjugation invariant norm $|| - ||$ on 1-PSs on $G$.

\begin{defn}[Unstable strata, blades and limit sets]
For $\beta = ([\lambda],m) \in X_*(G)/G \times \RR_{<0}$, we define the associated unstable stratum 
\[ S_\beta = \{ x \in X : \Lambda_x \cap [\lambda] \neq \emptyset \: \text{ and } \: M(x) = m \}. \]
If we fix a representative $\lambda$ of $[\lambda]$, then we define  \emph{limit sets} $Z_\beta^{ss}$ and \emph{blades} $Y_\beta^{ss}$ as follows
\[ Z_\beta^{ss} = \{x \in X^{\lambda(\GG_m)} : \lambda \in \Lambda_x \: \text{ and } \: M(x) = m \} \stackrel{p_\beta}{\longleftarrow} Y_\beta^{ss} = \{x \in X : \lambda \in \Lambda_x \: \text{ and } \: M(x) = m \}, \]
where $p_\lambda(x) := \lim_{t \ra 0} \lambda(t) \cdot x$. We define the index set $\cB = \{ 0 \} \sqcup \{ \beta : S_\beta \neq \emptyset \}$, which turns out to be finite (see Remark \ref{rmk inst strat finite} below).

Let us also introduce a closed subscheme $Z_\beta$ of the $\lambda$-fixed locus on which the normalised Hilbert--Mumford weight of $\beta$ is $m$ and its attracting set $Y_\beta$ under the flow by $\lambda(\GG_m)$ as $t \ra 0$:
\[ Z_\beta:= \left\{ x \in X^{\lambda(\GG_m)} : \frac{\mu(x,\lambda)}{||\lambda ||} =m \right\} \stackrel{p_\beta}{\longleftarrow} Y_\beta := \{ x \in X : \lim_{t \ra 0} \lambda(t) \cdot x  \in Z_\beta \}. \]
\end{defn}

Note that $Z_\beta^{ss} \subset Z_\beta$ and $Y_\beta^{ss} \subset Y_\beta$; we will soon see these are open subsets and thus we can give these sets a scheme structure. Furthermore, these schemes all depend on the chosen representative $\lambda$ of $[\lambda]$. Recall that $Y_\beta^{(ss)}$ and $Z_{\beta}^{(ss)}$ depend on a choice of 1-PS $\lambda \in [\lambda]$ as well as $m$, whereas $P_\lambda$ only depends on $\lambda$ (and not $m$). Note that in \cite{kirwan}, $P_\lambda$ is denoted by $P_\beta$.

Since the notion of adapted 1-PS depends on the choice of norm, the unstable strata also depend on this choice of norm, but $S_0 = X^{ss}$ does not.

\begin{prop}[Kirwan, {\cite[$\S$12]{kirwan}}]
Assume that $X \subset \PP^n$ is smooth. For an unstable index $\beta = ([\lambda],m) \neq 0 \in \cB$, the stratum $S_\beta$ can be described as follows.
\begin{equation}\label{iso strata}
 S_\beta = G Y_\beta^{ss} \cong G \times^{P_\lambda} Y_\beta^{ss}.
\end{equation}
Furthermore, the blades and limit sets can be described as follows.
\begin{enumerate}[label=\emph{\roman*)}]
\item The retraction $p_\beta : Y_\beta \ra Z_\beta$ is a Zariski locally trivial affine space fibration\footnote{This follows by work of Bia{\l}ynicki-Birula \cite{bb} describing the decomposition of a smooth projective variety with a $\GG_m$-action by taking the flow as $t \ra 0$. Here it is crucial that $X$ is \emph{smooth} for the fibres to be affine \emph{spaces}.}. 
\item The scheme $Y_\beta$ is preserved by the $P_\lambda$-action and $Z_\beta$ is preserved by the $L_\lambda$-action. Moreover, $p_\beta$ is equivariant with respect to the retraction $q_\lambda : P_\lambda \ra L_\lambda$.
\item For $p_\beta : Y_\beta \ra Z_\beta$, we have $p_\beta^{-1}(Z_\beta^{ss}) = Y_\beta^{ss}$.
\item $Z_\beta^{ss}$ is the GIT semistable set for the action of the reductive Levi subgroup $L_\lambda$ on $Z_\beta$ with respect to a canonical linearisation $\cL_\beta$ obtained by twisting by a rational multiple\footnote{Via $||-||$, we can identify characters and co-characters, so we twist by the rational character $\chi$ corresponding to the rational 1-PS $\frac{-m}{|| \lambda ||}\lambda$, so that $\mu^{\cL_\beta}(x,\lambda) = \mu(x,\lambda) + \langle \chi,\lambda \rangle = 0$ for $x \in X^\lambda$ to \lq cancel' the effect of $\lambda$.} of a character corresponding to $\lambda$. 
\end{enumerate}
\end{prop}

\begin{rmk}\label{rmk inst strat finite}
For $G$ acting linearly on $X = \PP(V)$ with a fixed choice of norm $|| -||$, we can compute the index set $\cB$ of the instability stratification in terms of the finitely many weights of the action of a maximal torus $T < G$. For any subset of the $T$-weights whose convex hull does not contain the origin (that is, this is a weight set of an unstable point), we let $\lambda$ be the primitive 1-PS of $T$ corresponding under $|| - ||$ to the ray in the $X_*(T)$ through the closest point to $0$ in the convex hull of this weight set and define $m$ to be the minimum normalised Hilbert--Mumford weight of any point with this weight set, then $\beta = ([\lambda],m) \in \cB$ (see \cite[Lemma 12.6]{kirwan}). In particular, $\cB$ is finite as there are only finitely many $T$-weights.
\end{rmk}

\begin{exer}[Instability stratification for binary forms]
Consider the action of $\SL_2$ on $\PP^d=\PP(k[x,y]_d)$ as in Exercise \ref{exer ss binary forms} and show for each integer $\frac{d}{2} < r \leq d$, there is an unstable stratum corresponding to binary forms $F(x,y)$ with a root of exactly multiplicity $r$.
\end{exer}

In the situation of a reductive group acting linearly on an affine space linearised by a character, using King's Hilbert--Mumford criterion \cite[Proposition 2.5]{king} (see also Remark \ref{rmk HM affine char}), one can construct an instability stratification \cite{hoskins_affinestrat}. For the action of $\GL_r$ on $\Mat_{r \times n}$ by left multiplication generalising Exercise \ref{exer grassmannain},  the (semi)stable locus for the character $\rho = \det$ is the matrices of maximal rank and the instability stratification is given by the rank (see \cite[Example 2.14]{hoskins_stackstrat}).

Given an instability stratification \eqref{instab strat}, we can ask if there is a categorical $G$-quotient of an unstable strata $S_\beta$, or equivalently via the isomorphism \eqref{iso strata}, a categorical $P_\lambda$-quotient of $Y_{\beta}^{ss}$.

\begin{prop}[Categorical quotients of unstable strata, {\cite[Lemma 3.1]{hoskins_kirwan}}]\label{prop cat quotient unstable strata}
The composition $Y_\beta^{ss} \stackrel{p_\beta}{\longrightarrow} Z_{\beta}^{ss} \stackrel{\pi}{\longrightarrow} Z_\beta/\!/_{\cL_\beta} L_\lambda$ of $p_\beta$ with the reductive GIT quotient $\pi$ is a categorical $P_\lambda$-quotient.
\end{prop}
\begin{proof}
This composition is $P_\lambda$-invariant, as $p_\beta$ is $q_\lambda$-equivariant and $\pi$ is $L_{\lambda}$-invariant. Given a $P_\lambda$-invariant morphism $f: Y_{\beta}^{ss} \ra S$, its restriction $f|$ to $Z_{\beta}^{ss}$ is $L_\lambda$-invariant and since $f$ is constant on orbit closures, we have $f = f| \circ \pi$. Then by the universal property of $\pi$, we see that $f|$ (and thus also $f$) factors uniquely via $Z_\beta/\!/_{\cL_\beta} L_\lambda$.
\end{proof}

Furthermore, by \cite[Lemma 3.1]{hoskins_kirwan} we have
\[ R(\overline{Y_\beta},\cL_\beta)^{P_\lambda} \cong R(Z_\beta,\cL_\beta)^{L_\lambda}, \]
which is finitely generated, and the categorical quotient coincides with the projective spectrum of the invariants. However, this categorical quotient factors via the retraction $p_\beta$ and so identifies every $x$ with $p_\beta(x) = \lim_{t \ra 0} \lambda(t) \cdot x$. Thus this categorical quotient is far from being an orbit space. Since the closed subscheme $Z_{\beta}^{ss}$ causes these identifications, we would like to further twist the linearisation to make $Z_{\beta}^{ss}$ become unstable and prevent these unwanted identifications. It is at this point where we see that the non-reductive action of $P_\lambda$ on ${Y^{ss}_\beta}$ is preferable to the reductive action of $G$ on $S_\beta$: the parabolic subgroup has more characters which can be used to twist the given ample linearisation; in $\S$\ref{sec quotient unstable strata}, we explain how to apply non-reductive GIT.

In fact, since linearisations give line bundles on quotient stacks, we have a line bundle
\[ \cL_\beta \ra [Y_\beta^{ss}/P_\lambda] \cong [S_\beta/G], \]
but whilst the corresponding $P_\lambda$-equivariant line bundle on $Y_\beta^{ss}$ is ample, the corresponding $G$-equivariant line bundle on $S_\beta$ is not ample (see \cite[Remark 9]{hoskins_kirwan}) and so is not suitable for working with from the perspective of GIT.

\section{Generalisations of reductive GIT to stacks}\label{sec stacks beyond GIT}

In $\S$\ref{sec moduli for stacks} we describe different types of moduli spaces for stacks and in $\S$\ref{sec exist crit}, we state a recent existence criterion \cite{AHLH} for stacks to admit a good moduli space.

Throughout this section, for simplicity, we will assume that our algebraically closed field $k$ is of characteristic $0$ to avoid the distinction between linearly reductive and geometrically reductive groups in positive characteristic, which in turn leads to a distinction between \emph{good moduli spaces} and \emph{adequate moduli spaces} for stacks. We will assume all stacks are noetherian algebraic stacks over $k$; however, everything in this section also extends to a relative setting. For a detailed introduction to algebraic spaces and stacks with a focus on moduli, see \cite[$\S$3]{AlperBook}.

\subsection{Moduli spaces for stacks}\label{sec moduli for stacks}

Associated to an action $G \times X \ra X$, there is a quotient stack $[X/G]$ whose points (and residual gerbes) describe the orbits (and stabilisers) of the action. A categorical quotient of this action is equivalent to a universal map from $[X/G]$ to a scheme. One can naturally ask if an arbitrary stack has a universal map to a scheme (or possibly an algebraic space, as was necessary in $\S$\ref{sec other methods const quotients}). If additionally, as in the definition of a coarse moduli space for a moduli functor, we ask for a bijection on $k$-points, this leads to the following notion.

\begin{defn}
A \emph{coarse moduli space} (CMS) for a stack $\fX$ is a map $\fX \ra X$ to an algebraic space which is initial for maps from $\fX$ to algebraic spaces and such that the induced map $|\fX(k)| \ra X(k)$ is bijective.
\end{defn}

Keel and Mori \cite{KeelMori} studied quotients of groupoids in the category of algebraic spaces and proved the existence of a categorical quotient under the assumption of finite stabilisers; let us state a stacky reformulation as in \cite{Conrad_KM}.

\begin{thm}[Keel--Mori Theorem]
An algebraic stack $\fX$ over $k$ with finite inertia stack admits a coarse moduli space $\pi: \fX \ra X$ with the following properties
\begin{enumerate}[label=\emph{\roman*)}]
\item $\cO_X \ra \pi_*\cO_{\fX}$ is  an isomorphism;
\item If $\fX$ is separated (resp. of finite type) over $k$, then $X$ is separated (resp. of finite type) over $k$;
\item $\pi$ is a proper universal homeomorphism;
\item Any flat base change of $\pi$ is also a coarse moduli space.
\end{enumerate}
\end{thm}

In particular, the Keel--Mori Theorem applies to separated Deligne--Mumford stacks. Requiring  $\fX \ra X$ to induce a bijection on closed points is very strong: as soon as there are non-closed orbits this condition fails. Fortunately Alper \cite{Alper_GMS} adapted the reductive GIT notion of good quotient to the setting of stacks as follows.

\begin{defn}
A \emph{good moduli space} (GMS) for a stack $\fX$ is a quasi-compact quasi-separated morphism $f: \fX \ra X$ to an algebraic space such that
\begin{enumerate}
\item the pushforward map $f_* :\cQ Coh(\fX) \ra \cQ Coh(X)$ on quasi-coherent sheaves is exact,
\item the natural map $\cO_X \ra f_* \cO_{\fX}$ is an isomorphism.
\end{enumerate}
\end{defn}

The following example shows this theory applies to quotients of linearly reductive groups.

\begin{ex}[Alper, {\cite[Example 12.9]{Alper_GMS}}]
For the classifying stack $BG = [\Spec k / G ]$ of an affine algebraic group, the map $\pi :  BG \ra \Spec k$ satisfies the second property in the definition of a good moduli space. Then $\pi_* : \cV ect^G_k \ra \cV ect_k$ is given by taking $G$-invariants of a linear representation of $G$. Hence, the first condition holds if and only if $G$ is linearly reductive.
\end{ex}

Subsequently, Alper later adapted his theory to include geometrically reductive groups in positive and mixed characteristic by giving a notion of an \emph{adequate moduli space} \cite{Alper_AMS}, which weakens the first property in the definition of good moduli spaces.

\begin{ex}[Alper, {\cite[Theorem 13.6]{Alper_GMS}}]
For a linearly reductive group $G$ acting on an affine scheme $X$, the quotient stack admits a good moduli space $[X/G] \ra X/\!/G$ given by the affine GIT quotient. More generally, for a linearly reductive group acting on a scheme $Y$ with respect to an ample linearisation $\cL$, the morphism $[Y^{ss}(\cL)/G] \ra Y/\!/_{\! \cL} G$ is a good moduli space.
\end{ex}

Let us note some important properties of good moduli spaces.

\begin{rmk} If $f: \fX \ra X$ is a good moduli space, then it has the following properties.
\begin{enumerate}
\item The morphism $f$ is initial among maps from $\fX$ to algebraic spaces \cite[Theorem 6.6]{Alper_GMS} and is surjective and universally closed \cite[Theorem 4.16 (i) and (ii)]{Alper_GMS}.
\item For every point $x \in X$, there is a unique closed point $x_0$ in $f^{-1}(x)$ and the automorphism group of $x_0$ is linearly reductive \cite[Theorem 9.1 and Proposition 12.14]{Alper_GMS}.
\item The morphism $f$ induces a bijection between closed points in $\fX$ and closed points in $X$ \cite[Theorem 4.16 (iv)]{Alper_GMS}. 
\item If $\fX$ is of finite type over $k$, then so is $X$ \cite[Theorem 4.16 (xi)]{Alper_GMS}.
\item Any base change of $f$ along a morphism of quasi-separated algebraic spaces is also a good moduli space \cite[Proposition 4.7]{Alper_GMS}.
\end{enumerate}
\end{rmk}

\subsection{Stability and existence criteria}\label{sec exist crit}

Halpern-Leistner \cite{HL_instab} and Heinloth \cite{Heinloth_HMstacks} studied how ideas in reductive GIT, such as the Hilbert--Mumford criterion, can be applied to stacks. The role of 1-PSs and their limits can be replaced by the stack \emph{Theta}:
\[ \Theta:= [\AA^1 /\GG_m] \]
over $\Spec \ZZ$, for the $\GG_m$-action on $\AA^1= \Spec k[x]$ by scalar multiplication. This stack plays a prominent role in the work of Halpern-Leistner \cite{HL_instab} and led to a notion of \emph{$\Theta$-stability} for stacks and a generalisation of GIT instability stratifications to \emph{$\Theta$-stratifications} of stacks \cite{HL_Theta}.

In this section, we will state a recent existence theorem of Alper, Halpern-Leistner and Heinloth \cite{AHLH} which gives necessary and sufficient conditions for a stack to admit a good moduli space. These conditions are valuative criteria known as $\Theta$-reductivity and S-completeness.

\begin{defn}[Valuative criteria for stacks]
A noetherian algebraic stack $\fX$ is said to be
\begin{enumerate}[label={\roman*)}]
\item \emph{$\Theta$-reductive} if for any DVR $R$, any morphism $\Theta_R \setminus \{ 0 \} \ra \fX$ extends uniquely to $\Theta_R$, where $\Theta_R= \Theta \times_{\ZZ} \Spec R$ and $0 \in \Theta_R$ denotes the unique closed point.
\item \emph{S-complete} if for any DVR $R$, any morphism $\overline{\mathrm{ST}}_R \setminus \{ 0 \} \ra \fX$ extends uniquely to 
\[ \overline{\mathrm{ST}}_R := [\Spec \left(R[s,t]/(st - \pi) \right)/\GG_m ] \]
for a uniformiser $\pi$ and $\GG_m$-action with weights $+1,-1$ on $s,t$.
\end{enumerate}
\end{defn}

The stack $\overline{\mathrm{ST}}_R$ originates from work of Heinloth \cite{Heinloth_HMstacks} and naturally generalises Example \ref{ex wts plus minus 1} where $\GG_m$ acts on $\AA^2$ with weights $+1,-1$. Recall that after removing the origin,  $\AA^2 \setminus \{ 0 \}$ has non-separated geometric quotient given by the affine line with two origins. S-completeness should be thought of as a stacky valuative criterion for separatedness.

\begin{rmk}
Assume that $\fM$ is a moduli stack for objects in an abelian category as in \cite[$\S$7]{AHLH}. Then a morphism $\Theta_k \ra \fM$ is a (weighted) filtration on a family of $\fM$ over $k$ such that the associated graded lies in $\fM$. Let $R$ be a DVR with fraction field $K$ and residue field $k = R /(\pi)$, where $\pi$ denotes a uniformiser. In this case, we can interpret the above conditions as follows.
\begin{enumerate}
\item A morphism $\Theta_R \setminus \{ 0 \} \ra \fM$ is given by a family of $\fM$ over the DVR $R$ together with a filtration on the generic fibre $K$ (whose associated graded object lies in $\fM$). This extends uniquely to $\Theta_R$ if and only if the filtration on the generic fibre extends uniquely to the special fibre (again with the associated graded object lying in $\fM$).
\item A morphism $\overline{\mathrm{ST}}_R \setminus \{ 0 \} \ra \fM$ is equivalent to two families of $\fM$ over $R$ whose generic fibres over $K$ are equivalent. This extends uniquely to $\overline{\mathrm{ST}}_R$ if and only if the special fibres have filtrations whose associated graded objects are isomorphic\footnote{For the stack of semistable vector bundles on a curve, this asks for the bundles on the special fibre to be S-equivalent.} in $\fM$.
\end{enumerate}
This can be seen from looking at the following diagrams appearing in \cite[$\S$6.7.2]{AlperBook}
\[ \xymatrix@=1em{ & \Spec R \ar@{^{(}->}[rd]_{x \neq 0} & & B \GG_{m,R} \ar@{_{(}->}[ld]^{x = 0} & \\ \Spec K \ar@{^{(}->}[ru] \ar@{^{(}->}[rd] & & \Theta_R & & B\GG_{m,k}, \ar@{^{(}->}[ld] \ar@{_{(}->}[lu] \\ & \Theta_K \ar@{^{(}->}[ru]^{\pi \neq 0}  & & \Theta_k \ar@{_{(}->}[lu]_{\pi = 0} &  } \]
where the left side of the diagram corresponds to the open immersion $\Theta_R \setminus \{ 0 \} \hookrightarrow \Theta_R$ and all morphisms on the left are open immersions, and the right side represents the closed immersion $\{ 0 \} \hookrightarrow \Theta_R$ and all morphisms on the right are closed immersions.

There is a similar diagram for $\overline{\mathrm{ST}}_R$:
\[ \xymatrix@=1em{ & \Spec R \ar@{^{(}->}[rd]_{s \neq 0} & & \Theta_k \ar@{_{(}->}[ld]^{s = 0} & \\ \Spec K \ar@{^{(}->}[ru] \ar@{^{(}->}[rd] & & \overline{\mathrm{ST}}_R & & B\GG_{m,k}. \ar@{^{(}->}[ld] \ar@{_{(}->}[lu] \\ & \Spec R \ar@{^{(}->}[ru]^{t \neq 0}  & & \Theta_k \ar@{_{(}->}[lu]_{t = 0} &  } \]
\end{rmk}

We can now state the existence theorem of Alper--Halpern-Leistner--Heinloth \cite{AHLH}.

\begin{thm}[Existence criterion for GMS, {\cite[Theorem A]{AHLH}}]\label{thm AHLH exist}
Let $\fX$ be an algebraic stack of finite type over $k$ of characteristic zero, with affine diagonal. Then $\fX$ admits a separated good moduli space $X$ if and only if $\fX$ is $\Theta$-reductive and S-complete. Moreover, $X$ is proper if and only if $\fX$ satisfies the existence part of the valuative criterion for properness.
\end{thm}

In characteristic zero, the existence criterion is much nicer due to the existence of \emph{\'{e}tale local quotient presentations} around closed points with reductive stabiliser due to Alper--Hall--Rydh \cite{AlperHallRydh}, which is a stacky generalisation of Luna's \'{e}tale slice theorem \cite{luna}. Showing these conditions are necessary for $\fX$ to have a good moduli space is relatively formal. To show these conditions suffice to construct a good moduli space, one first uses S-completeness to show closed points have reductive stabilisers and then one glues the affine GIT quotients associated to the \'{e}tale local quotient presentations, which uses both $\Theta$-reductivity and S-completeness.

In positive characteristic there is also an existence theorem for adequate moduli spaces, but as an input it requires the existence of \'{e}tale local quotient presentations, which is part of the \emph{local reductivity} assumption in \cite[Theorem A]{AHLH}.

With this existence criterion in hand, one can construct moduli spaces without GIT as follows.
\begin{enumerate}
\item Interpret the moduli problem as an algebraic stack $\fM$.
\item Apply an existence theorem to obtain a (proper) good moduli space $\fM \ra M$. 
\end{enumerate}
However, this only yields a proper good moduli space in cases where GIT would provide a projective moduli space. Consequently, one further step is required for this strategy:
\begin{enumerate}
\item[(3)] Find an ample line bundle on $M$ to show it is projective.
\end{enumerate}
This approach has been implemented for moduli of smooth projective curves in \cite{ChengLian}, where the second step uses the Keel--Mori Theorem and the third step follows Kollar's proof of projectivity using the determinant of a relative pluricanonical sheaf for the universal family; and for moduli of vector bundles on a smooth projective curve in characteristic zero in \cite{AlperBelmans}, where the second step uses Theorem \ref{thm AHLH exist} and the third step follows Faltings' proof of projectivity using a determinantal line bundle constructed from the universal family. For moduli of representations of an acyclic quiver (in arbitrary characteristic), a new moduli-theoretic proof of projectivity was given in \cite{BDFHMT}, where again a determinantal line bundle is used. Of course, in all these cases, reductive GIT could be applied to produce the moduli space without this extra work! However, this intrinsic moduli-theoretic approach to projectivity does give new insights: the line bundle obtained in the final step is of inherent interest to the moduli problem and the techniques give new effective bounds for global generation of determinantal line bundles (for example, see \cite[Theorem B]{BDFHMT}).

The advantage of this approach is that one does not need a quotient presentation of the stack; for example, for moduli of Bridgeland semistable objects in a derived category without a quotient presentation, this approach can be applied provided the stack is algebraic as in \cite[Theorem 1.3]{Toda}. There have been impressive applications to moduli of K-polystable Fano varieties \cite{ABHLX_Kstab} and moduli of torsors for a Bruhat–Tits group scheme over a curve \cite[Theorem 8.1]{AHLH}.

Since the theory of good (and adequate) moduli spaces is based on GIT for reductive groups, there are still many interesting moduli problems which do not admit good (or adequate) moduli spaces, such as moduli stacks of weighted projective hypersurfaces, where non-reductive groups naturally appear (see $\S$\ref{sec weighted proj hyp}).

\section{Non-reductive geometric invariant theory}

In this section, we will describe recent progress on non-reductive GIT (in characteristic zero) due to B\'{e}rczi, Doran, Hawes and Kirwan \cite{BDHK,BDHK_proj} concerning non-reductive groups with \emph{graded unipotent radicals}. Before we describe their approach, we start by illustrating the issues in constructing non-reductive quotients in $\S$\ref{sec additive bad} and give a survey of some previous contributions in $\S$\ref{sec survey non-red}. We then proceed to explain the relationship between $\GG_a$-actions and locally nilpotent derivations in $\S$\ref{sec Ga LND}, as well as their interaction with $\GG_m$-actions in $\S$\ref{sec Ga Gm}. In $\S$\ref{sec gr uni gp} we introduce graded unipotent groups and state the main results of \cite{BDHK,BDHK_proj} in $\S$\ref{sec statements nrGIT}, before giving details on some of the proofs in $\S$\ref{sec proof Uhat thm}. Throughout this section, we assume $k$ is of characteristic zero.

\subsection{Examples of bad behaviour of additive actions}\label{sec additive bad}

The ring of invariants being non-finitely generated for a non-reductive group is not the only issue that arises when trying to construct quotients of non-reductive group actions. Even when the ring of invariants is finitely generated (for example, see Theorem \ref{thm weitzenbock}), there can be further issues:
\begin{enumerate}
\item The quotient morphism given by the inclusion of invariants may fail to be surjective (and in general, its image is only a constructible subset), so will not be a good quotient.
\item There may not be enough invariants to separate disjoint closed orbits in contrast to the case for geometrically reductive groups (Lemma \ref{geom red lem}).
\item Invariants may not extend to the ambient space.
\end{enumerate}

Let us give some concrete examples of linear $\GG_a$-actions that demonstrate these issues.

\begin{ex} The following examples are all built from powers (or symmetric powers) of the standard representation of the additive group on $\AA^2$ given by $\GG_a < \SL_2$ (as the upper triangular unipotent radical) acting on $V=\AA^2$ by left multiplication; thus Theorem \ref{thm weitzenbock} applies.
\begin{enumerate}
\item For $\GG_a$ acting on $V \times V = \AA^4$ via $u \cdot (x_1,x_2,x_3,x_4) = (x_1 + u x_2, x_2, x_3 + u x_4, x_4)$, 
\[ \cO(V \times V)^{\GG_a} = k[x_2,x_4,x_1x_4 - x_2 x_3] \]
and the (invariant theoretic) quotient map $\pi : \AA^4 \ra \AA^3 \cong \spec \cO(V \times V)^{\GG_a}$ only has constructible image, as it misses the punctured line $\{ (0,0,\eta) : \eta \neq 0 \}$.
\item For $\GG_a$ acting on $\Sym^2(V) = \AA^3$ via $u \cdot (x_1,x_2,x_3) = (x_1 + 2u x_2 + u^2 x_3, x_2 + u x_3, x_3)$, 
\[ \cO(\Sym^2(V))^{\GG_a} = k[x_3,x_2^2 - x_1x_3] \]
and these invariants do not separate the closed orbits $\GG_a \cdot (1,1, 0 ) = \{ (\eta,1,0)\}$ and $\GG_a \cdot (1,-1, 0 ) = \{ (\eta,-1,0)\}$.
\item For $\GG_a$ acting on $\Sym^2(V)=\AA^3$ as in (2), the $\GG_a$-invariant closed subscheme $X = \{ x_3 = 0 \}$ has an invariant function $x_2 \in \cO(X)^{\GG_a}$ which does not extend to $\cO(\AA^3)^{\GG_a}$.
\end{enumerate}
\end{ex}

In particular, just trying to get rings of invariants to be finitely generated will not suffice to generalise the nice properties of reductive GIT. 

Even in the case of free $\GG_a$-actions, there are examples which do not admit a geometric quotient (for example, see \cite[Example 18]{DerksenFreeGa} which is proved via a geometrical argument).

\subsection{Short historical note on non-reductive group actions}\label{sec survey non-red}

Let us summarise some of the contributions towards the development of non-reductive GIT, before turning to \cite{BDHK,BDHK_proj}.

As mentioned in Remarks \ref{rmk transfer principle} and \ref{rmk Grosshans subgp}, the transfer principle was used by Grosshans \cite{Grosshans_obs} to prove finite generation in certain cases. Grosshans proved the unipotent radical $U$ of a parabolic subgroup in a reductive group $G$ is a Grosshans group and thus if a $U$-action on an affine variety extends to $G$, then the ring of $U$-invariant functions is finitely generated \cite{Grosshans_inv}; his proof uses a grading by weights of a maximal torus (see \cite[Theorem 2.7]{brion}), which will also play an important role in non-reductive GIT.

Fauntleroy defines global intrinsic notions of semistability for a connected unipotent group action on a quasi-affine normal variety in terms of properties of invariant sections and showed if the stabilisers are finite that the open semistable set admits a categorical quotient \cite[Theorem 5]{Fauntleroy_first}. In subsequent work \cite{Fauntleroy}, he more generally defined a notion of properly stable points for a linearised action of a connected linear algebraic group on a normal projective variety and shows this locus has a quasi-projective geometric quotient. Fauntleroy combined ideas from reductive GIT with Seshadri covers, which Seshadri used to show actions of connected linear algebraic groups on a normal variety with finite stabiliser groups admit geometric quotients up to a finite equivariant replacement \cite{seshadri_quotients}.

For non-reductive actions on an affine scheme, Winkelmann \cite{Winkelmann} showed there is a rational quotient map to a quasi-affine variety, whose coordinate ring is the ring of invariants (and is not necessarily finitely generated). In fact, he showed the study of coordinate rings of quasi-affine varieties corresponds to the study of rings of invariants for $\GG_a$-actions on affine varieties. 

Alternatively, one can ignore the issue of whether or not rings of invariants are finitely generated and work in the category of all schemes (not necessarily of finite type) over $k$; Greuel and Pfister take this approach in \cite{GreuelPfister} to define a notion of stability for unipotent group actions and show there is a geometric quotient in the category of varieties. Their motivation came from the study of singularities, where non-reductive translation actions appear.

We note that additive group actions also arise as translation actions in affine geometry and have led to progress on classical questions on affine spaces such as the cancellation problem, the existence of exotic affine spaces and the Jacobian conjecture (for example, see \cite{Kraft_survey}).

Doran and Kirwan \cite{dorankirwan} give various notions of (semi)stability for non-reductive GIT using properties of invariants and by transferring the problem to a reductive GIT setting using a notion of \lq fine reductive envelope' and they obtain various types of quotients of these (semi)stable sets.

The recent progress on non-reductive GIT \cite{BDHK}, which we explain below, uses a multiplicative group to grade the unipotent radical; this enables the construction of geometric unipotent quotients, as well as providing a natural projective completion and an explicit Hilbert--Mumford type description of stability. To trace back the origin of this grading multiplicative group, we begin with the correspondence between $\GG_a$-actions and locally nilpotent derivations, and will see these multiplicative group actions naturally appear when there is a slice of the $\GG_a$-action.

\subsection{Actions of the additive group}\label{sec Ga LND}

Since $k$ is of characteristic zero, we can utilise the dictionary between additive group actions and locally nilpotent derivations (e.g.\ see \cite{Freudenburg_LND}).

For an action $\sigma : \GG_a \times X \ra X$ on an affine scheme, the coaction
\[ \sigma^* : \cO(X) \ra \cO(\GG_a \times X) \cong \cO(X) \otimes_k k[t] \]
can be used to define a derivation $D_\sigma: \cO(X) \ra \cO(X)$ given by $D_\sigma(f) := \frac{\partial}{\partial t} (\sigma^*(f))|_{t = 0}$ satisfying the Leibniz rule $D_\sigma(fg) = fD_\sigma(g) + D_\sigma(f)g$. This derivation is \emph{locally nilpotent} (i.e.\ for any $f \in \cO(X)$, there is $n \in \NN$ such that $D_\sigma^n(f) = 0$), as one can inductively show
\[ \sigma^*(f) = \sum_{n \geq 0} \frac{D_\sigma^n(f) t^n}{n!} \]
and since the left side is a polynomial in $t$, we must have $D_\sigma^n(f) = 0$ for all $n$ sufficiently large.

Conversely, given a locally nilpotent derivation $D : A \ra A$ on the $k$-algebra $A = \cO(X)$ of functions on an affine scheme, we can exponentiate to construct a coaction
\[ \sigma^*_D := \exp(tD) = \sum_{n \geq 0} \frac{D^n t^n}{n!} : A \ra A[t] \]
which is well-defined, as $D$ is locally nilpotent. 

These constructions are inverse to each other and we collect some other results relating these geometric and algebraic points of view.

\begin{prop}\label{prop Ga LND slice}
For an affine variety $X$ with coordinate ring $A=\cO(X)$, there is a bijective correspondence
\[ \begin{array}{ccc}
\{ \, \mathrm{ actions } \:\: \GG_a \times X \ra X \} & \longleftrightarrow & \{ \: \mathrm{ locally \: nilpotent \: derivations } \:  D: A \ra A \} \\
\sigma : \GG_a \times X \ra X & \mapsto & D_{\sigma} := \frac{\partial}{\partial t} (\sigma^*(-))|_{t = 0} \\
\sigma^*_D = \exp(tD) : A \ra A[t] & \mapsfrom & D : A \ra A.
\end{array}
\]
For an action $\sigma$ corresponding to a locally nilpotent derivation $D$, the following statements hold.
\begin{enumerate}[label=\emph{\roman*)}]
\item The ring of invariants is the kernel of the derivation: $\cO(X)^{\GG_a} = \ker(D)$,
\item $x \in X$ is $\GG_a$-fixed if and only if $D(A) \subset \fm_x$.
\item If $D$ has a slice (i.e.\ there exists $s \in A = \cO(X)$ such that $D(s) = 1$), then $\cO(X)^{\GG_a}$ is a finitely generated $k$-algebra, the subscheme $S := \{ s = 0 \} \subset X$ is a geometric slice of the $\GG_a$-action and $X \ra \Spec \cO(X)^{\GG_a} $ is a trivial principal $\GG_a$-bundle.
\end{enumerate}
\end{prop}
\begin{proof} Let us just give some details on the third statement, as this will be important in what follows. Suppose that $s \in \cO(X)$ is a slice for $D$. Then define a $k$-algebra homomorphism 
\[ \Phi : A \ra A, \quad f \mapsto  \exp(tD(f))|_{t = -s} \]
such that $\im(\Phi) = \ker(D)$. In particular, $\cO(X)^{\GG_a} = \ker(D) = \im(\Phi)$ is a finitely generated $k$-algebra: the images under $\Phi$ of generators for the algebra $A=\cO(X)$ are generators of $\cO(X)^{\GG_a}$. 

By induction, any $f \in \cO(X)$ is a polynomial in $s$ with coefficient in $\im(\Phi) = \cO(X)^{\GG_a}$:
\[ f = \sum_{n \geq 0 } \frac{\Phi(D^n(f)) s^n}{n!} \]
and thus $\cO(X) = \cO(X)^{\GG_a} [s]$ and $\cO(X)^{\GG_a} = \cO(X)/(s)$. Moreover $S = \{ s = 0 \} \subset X$ is isomorphic to $\Spec \cO(X)^{\GG_a}$. We claim that $S$ is a geometric slice; that is $\GG_a \times S \ra X$ given by $(u,x) \ra \sigma(u,x)$ is an isomorphism. Indeed, by considering the following commutative diagrams:
\[ \xymatrix{\cO(X) \ar[rr]^{\Phi} \ar@{->>}[rd] & & \cO(X)  & & & X & & X \ar[ll]_{\Phi^*} \ar[dl] \\
& \cO(X)^{\GG_a} \ar@{^{(}->}[ru] & & & & & S \ar@{_{(}->}[lu] }\]
one can construct the inverse.
\end{proof}

Note that if there is a slice, then the coordinate ring is a polynomial ring in a single variable with coefficients in the ring of invariants. In particular, the coordinate ring is naturally graded by $\NN$ (the degree of this polynomial) and thus this gives a $\GG_m$-action. This naturally leads to considering actions of semi-direct products of $\GG_a$ and $\GG_m$.

\subsection{Semi-direct products of additive and multiplicative groups}\label{sec Ga Gm}

For an affine scheme $X$, a $\GG_m$-action on $X$ is equivalent to a $\ZZ$-grading of the $k$-algebra $\cO(X)$. The limit under the $\GG_m$-action as $t \ra 0$ exists for all $x \in X$ if and only if $\GG_m$ acts with non-negative weights on $X$, which is if and only if the grading on $\cO(X)$ is supported in non-positive degrees.

\begin{defn}
A semi-direct product of $\GG_a$ and $\GG_m$ is given by specifying a group homomorphism $\varphi : \GG_m \ra \Aut(\GG_a)$ and defining the semi-direct product $\GG_a \rtimes_\varphi \GG_m$ to have underlying set $\GG_a \times \GG_m$ with group operation
\[ (u,t) \cdot (v,s) = (v + \varphi_s(u),ts). \]
For $n \in \ZZ$, let us define $\varphi$ by $ t \mapsto t^{-n}$ and write $\GG_a \rtimes_n \GG_m$ for the associated semi-direct product, where $t u t^{-1} = t^n u$.
\end{defn}

\begin{ex}
The upper triangular Borel $B < \SL_2$ is isomorphic to a semi-direct product $\GG_a \rtimes_2 \GG_m$, via $(u,t) \mapsto \left( \begin{smallmatrix} t & ut \\ 0 & t^{-1} \end{smallmatrix} \right).$
\end{ex}

\begin{rmk}
For an action of a semi-direct product $\GG_a \rtimes_n \GG_m$ on an affine scheme, the locally nilpotent derivation $D$ associated to the $\GG_a$-action is homogeneous of degree $n$ with respect to the grading $\cO(X) = \bigoplus_{r \in \ZZ} \cO(X)_r$ determined by the $
\GG_m$-action; that is, $D(\cO(X)_r) \subset \cO(X)_{r+n}$.
\end{rmk}

The following result is the key proposition which enables the inductive construction of quotients of unipotent group actions in the presence of an appropriate $\GG_m$-action grading the unipotent action. This is a modification of \cite[Lemma 7.3]{BDHK_proj}.

\begin{prop}[Key Proposition]\label{key prop}
For an affine $\GG_a$-scheme $X$ with locally nilpotent derivation $D : \cO(X) \ra \cO(X)$, the following statements are equivalent:
\begin{enumerate}
\item $D$ has a slice (i.e.\ there exists $s \in \cO(X)$ with $D(s) = 1$),
\item The $\GG_a$-action extends to $\GG_a \rtimes_n \GG_m$ for some $n >0$ such that
\begin{enumerate}
\item  $\lim_{t \ra 0} t \cdot x$ exists for all $x \in X$ and
\item $\Stab_{\GG_a}(z) = \{ e \}$ for all $z \in Z:=\{ \lim_{t \ra 0} t \cdot x\}$.
\end{enumerate}
\end{enumerate}
In particular, if (2) holds, then there is a trivial $U$-quotient $X \mapsto \cO(X)^{\GG_a}$.
\end{prop}
\begin{proof}
Suppose $D$ has a slice $s$; then $\cO(X) = \cO(X)^{\GG_a}[s]$ as in the proof of Proposition \ref{prop Ga LND slice}. Since $\cO(X)$ is a polynomial ring in $s$, it is naturally graded, and we choose a $\ZZ_{\leq 0}$-grading given by $\cO(X)_{-r} = \cO(X)^{\GG_a}s^r$ so that $\lim_{t \ra 0} t \cdot x$ exists for all $x \in X$. Since $D(s) = 1$, we see that $D$ is homogeneous of degree $1$ and thus there is an action of $\GG_a \rtimes_1 \GG_m$. Since $ X \ra \spec \cO(X)^{\GG_a}$ is a trivial $\GG_a$-bundle (see Proposition \ref{prop Ga LND slice}), all $\GG_a$-stabilisers are trivial.

Let us outline the converse direction (see \cite[Lemma 7.3]{BDHK_proj} for the full proof). Given $n > 0$ and a $\GG_a \rtimes_n \GG_m$-action on $X$ such that $\lim_{t \ra 0} t \cdot x$ exists for all $x \in X$ and $\Stab_{\GG_a}(z) = \{ e \}$ for all $z \in Z$, we note that $D$ has homogeneous degree $n$ with respect to the grading $\cO(X) = \oplus_r \cO(X)_r$ given by the $\GG_m$-action. By the assumption that $\lim_{t \ra 0} t \cdot x$ exists for all $x \in X$, this grading is supported in non-positive degrees. We have
\begin{equation}\label{eq D}
 D(\cO(X)_r) \subset \left\{ \begin{array}{ll} 0 & \mathrm{ if } \: r > -n \\ \cO(X)_0 & \mathrm{ if } \: r = -n \\ \cO(X)_{<0} & \mathrm{ if } \: r < -n \end{array}\right.
\end{equation}
and we will show that $D(\cO(X)_{-n})=\cO(X)_0$, so there exists $s \in \cO(X)_{-n}$ with $D(s) = 1$. Let
\begin{equation}\label{def I}
 I := D(\cO(X)_{-n}) \oplus \cO(X)_{<0},
\end{equation}
which is $\GG_a$-stable by \eqref{eq D} and also $\GG_m$-stable, as $\cO(X)_{<0}$ is a sum of $\GG_m$-weight spaces and the $\GG_m$-action on $D(\cO(X)_{-n}) \subset \cO(X)_0$ is trivial. Then it suffices to show that $I$ is an ideal (Claim 1) and moreover $I = \cO(X)$ (Claim 2). Indeed, as $I= \cO(X)$ we must have $D(\cO(X)_{-n}) = \cO(X)_0$ and so $D(s) = 1$ for some $s \in \cO(X)_{-n}$.

To prove Claim 1, we need to show for $f \in I$ and $h \in \cO(X)$ that $ h f \in I$; for this we can write $h = \sum h_r$ with respect to the $\GG_m$-grading and it suffices to show $h_rf \in I$ for all $r$. We can also write $f = D(p_{-n}) + \sum_{r <0} f_r$ by \eqref{def I}. The only non-trivial case is for $h = h_0$ and $f = D(p_{-n})$; however, by the Leibniz rule $D(h_0p_{-n}) = h_0D(p_{-n})$ and so $h_0 f = D(h_0p_{-n}) \in I$.

To prove Claim 2, we argue by contradiction. If $ I \subsetneq \cO(X)$, then it is contained in a maximal ideal $\mathfrak{m}_x$. Since $I$ is $\GG_a \rtimes_n \GG_m$-stable, so is $\fm_x$ and so it corresponds to a $\GG_a \rtimes_n \GG_m$-fixed point $x \in Z$. However, this contradicts the assumption that $\Stab_{\GG_a}(z) = \{ e \}$ for all $z \in Z$.
\end{proof}

Note that there is a choice of sign here. One could consider actions of $\GG_a \rtimes_n \GG_m$ for $n < 0$ such that $\lim_{t \ra \infty} t \cdot x$ exists for all $x \in X$ and obtain an analogous result.

\begin{rmk}
For a linear representation $\GG_a \rtimes \GG_m \ra \GL(V)$, we have $V_{\max} \subset V^{\GG_a}$, where $V_{\max}$ is the weight space for the maximal $\GG_m$-weight.
\end{rmk}

\begin{ex}
Consider the upper triangular Borel $\GG_a \rtimes_2 \GG_m \cong B < \SL_2$ acting on $V=\AA^2$ by left multiplication. We have 
\[ V = V_{+1} \oplus V_{-1} = \{ (*,0)\} \oplus \{ (0,*)\}\]
and $V_{\max} = V_{+1} \subset V^{\GG_a}$.
\end{ex}

\subsection{Graded unipotent groups}\label{sec gr uni gp}

Generalising Proposition \ref{key prop}, we will consider actions of unipotent groups which are graded by a $\GG_m$-action in the following sense.

\begin{defn}
A \emph{graded unipotent group} is a semi-direct product $\hU:= U \rtimes \GG_m$ of a unipotent group with a multiplicative group such that the conjugation of $\GG_m$ on the Lie algebra of $U$ has strictly positive weights.
\end{defn}

\begin{ex}
For $n >0$, the group $\GG_a \rtimes_n \GG_m$ is a graded unipotent group. In particular, the upper triangular Borel $B < \SL_2$ is a graded unipotent group.
\end{ex}

\begin{prop}\label{key prop U}
Let $X$ be an affine scheme with an action of a graded unipotent group $\hU:= U \rtimes \GG_m$ such that $\lim_{t \ra 0} t \cdot x$ exists for all $x \in X$ and $\Stab_U(z) = \{ e \}$ for all $z \in Z:=\{ \lim_{t \ra 0} t \cdot x\}$, then $\cO(X)^U$ is finitely generated and $X \ra \Spec \cO(X)^U$ is a trivial $U$-quotient.
\end{prop}
\begin{proof}
The idea is to iteratively apply Proposition \ref{key prop} as in the proof of \cite[Proposition 7.4]{BDHK_proj}. By lifting a filtration on the Lie algebra via the exponential map, we obtain normal subgroups 
\[ \{ e \} = U_0 < U_1 < \cdots < U_r  = U \]
whose successive quotients are copies of $\GG_a$ on which $\GG_m$ acts by conjugation with strictly positive weights. Assume that we have constructed a quotient $q_j: X \ra X_j:= \Spec \cO(X)^{U_j}$ which is a trivial $U_j$-quotient. The base case for $j = 1$ is given by Proposition \ref{key prop}. For the inductive step, we claim that we can apply Proposition \ref{key prop} to the $\GG_a \cong U_{j+1}/U_j$ action on $X_j$ to show there is a trivial $U_{j+1}/U_j$-quotient $X_{j} \ra \cO(X_j)^{U_{j+1}/U_j}$. Then the composition 
\[ X \ra X_j:= \Spec \cO(X)^{U_j} \ra \cO(X_j)^{U_{j+1}/U_j} = \cO(X)^{U_{j+1}}\]
is a principal $U_j$-bundle (see \cite[Proposition 4.7]{BK_momentmap}), which is trivial as the base is affine by \cite[Theorem 3.12]{AsokDoran}.

To complete the proof, we must show that for the $\GG_a \cong U_{j+1}/U_j$ action on $X_j$ graded by $\GG_m < \hU$ all limits $\lim_{t \ra 0} t \cdot x_j$ exists for $x_j \in X_j$ and $\Stab_{U_{j+1}/U_j}(x_j) = \{ e \}$ for all $x_j \in X_j^{\GG_m}$ in order to be able to apply Proposition \ref{key prop}. Since $q_j$ is $\GG_m$-equivariant, we have
\[ \lim_{t \ra 0} t \cdot q_j(x) = \lim_{t \ra 0} q_j(t \cdot x) = q_j\left( \lim_{t \ra 0} t \cdot x \right). \]
Thus the set $Z_j$ of such limits in $X_j$ is contained in $q_j(Z)$. For a point $q_j(z) \in Z_j$, suppose that $uU_j \in \Stab_{U_{j+1}/U_j}(q_j(z))$; that is, there exists $u' \in U_j$ such that $u'z = uz$. However, by assumption $\Stab_U(z) = \{ e \}$ and so we have $u = u' \in U_j$, which means $\Stab_{U_{j+1}/U_j}(q_j(z))$ is trivial and we can apply Proposition \ref{key prop} to $X_j$ as claimed.
\end{proof}

\begin{rmk} 
For a representation $\rho : \hU \ra \GL(V)$, note that $V_{\max} \subset V^{U}$.
\end{rmk}

In particular, rather than trying to find $U$-invariant sections, we can use $\GG_m$-maximal sections to construct a non-reductive GIT quotient as in the next subsection.

\subsection{Statement of the key results in non-reductive GIT}\label{sec statements nrGIT}

Instead of working with a linearised action on a projective scheme, for simplicity we will assume that we have a linear action on $X=\PP(V)$ as in $\S$\ref{sec proj GIT}. As before, in the case of a (very ample) linearisation $\cL$ on $X$, we get a projective embedding $X \hookrightarrow \PP(V)$ where $V = H^0(X, \cL)^*$ has a linear action. 

This concerns actions of affine algebraic groups, whose unipotent radical is graded by a $\GG_m$.

\begin{defn}
Let $G = U \rtimes R$ be an affine algebraic group with unipotent radical $U$ and reductive Levi factor $R$. A 1-PS $\lambda : \GG_m \ra Z(R)$ is said to \emph{grade} $U$ if its conjugation action on $\Lie U$ has strictly positive weights. In this case, we say \emph{$G$ has graded unipotent radical (by $\lambda$)}.
\end{defn}

This is what is called an \emph{internal grading} in \cite{BDHK_proj}, where there is also a more general notion of an \emph{external grading}; however, we will stick to the simpler internal point of view, as this suffices in all examples and applications we consider. 

There could be several central 1-PSs in $Z(R)$ that grade $U$, but we will assume we have fixed a grading $\GG_m$ and write $\hU = U \times \GG_m$. Different grading 1-PSs gives rise to different quotients, so can be thought of as an additional choice to the linearisation (see Remark \ref{rmk adapted} below). 

\begin{ex}
A parabolic subgroup $ P$ of $\GL_n$ (or more generally any reductive group) has graded unipotent radical: write $P = P_\lambda$ for a 1-PS $\lambda$, then $P_\lambda = U_\lambda \rtimes L_\lambda$, with $\lambda$ being central in $L_\lambda$ and grading $U_\lambda$.
\end{ex}

\begin{defn}[Minimal weight space and attracting set]
For a multiplicative group $\GG_m$ acting linearly on $X = \PP(V)$, we define the minimal weight space $Z_{\min}$ and minimal attracting set $X_{\min}$ as follows:
\[ \begin{array}{rlcrl}
Z_{\min} & := \: \PP(V_{\min}) & \stackrel{p}{\longleftarrow} &  X_{\min} & := \: \{ x \in X : \lim_{t \ra 0} t \cdot x \in Z_{\min} \}
\\ 
& \: =   \{ x : \wt_{\GG_m}(x) = \{ \omega_{\min} \} \}  & & & \:  = \: \{ x : \omega_{\min} \in \wt_{\GG_m}(x) \}
\end{array}
\]
where $\omega_{\min} = \omega_0 < \omega_1 < \dots < \omega_n $ are the $\GG_m$-weights on $V$ and $V_{\min} = V_{\omega_{\min}}$.
\end{defn}

In \cite{BDHK}, the minimal attracting set $X_{\min}$ is denoted $X_{\min}^0$, but we have simplified the notation. For the associated Bia{\l}ynicki-Birula stratification \cite{bb} (flowing as $t \ra 0$), the variety $X_{\min}$ is the open stratum and $p : X_{\min} \ra Z_{\min}$ is a Zariski locally trivial affine space fibration.

In \cite{BDHK_proj}, there are two types of assumptions needed for the construction of non-reductive GIT quotients: the first (Definition \ref{def adapted}) concerns the positioning of the weights for the linearised action of the grading $\GG_m$, which selects a particular VGIT chamber for $\GG_m$ and can be achieved by twisting the linearisation by a (rational) character $\chi : \hU \ra \GG_m$, and the second (Assumption \ref{condU0} in Definition \ref{def stab ass}) requires certain unipotent stabiliser groups to be trivial, which is referred to as \emph{semistability coincides with stability} in \cite{BDHK_proj}.

\begin{defn}[Adapted linearisation]\label{def adapted}
Let $ G = U \rtimes R$ be a group with unipotent radical graded by $\GG_m < Z(R)$ acting linearly on $X = \PP(V)$. We say this linearised action is \emph{adapted} if the $\GG_m$-weights on $V$ satisfy
\[ \omega_{\min} = \omega_0 < 0 < \omega_1 < \dots < \omega_n. \]
\end{defn}

\begin{rmk}\label{rmk adapted}
The assumption that the linearised action is adapted fixes a particular VGIT chamber for the $\GG_m$-action, where (semi)stability is given by
\[ X^{\GG_m-(s)s} = X_{\min} \setminus Z_{\min}. \]
Since twisting the linearisation by a character $\chi : \hU \ra \GG_m$ shifts the weights by $-\chi$ (where we identify characters of $\hU$ with integers $\ZZ$ such that $1 \in \ZZ$ corresponds to a character $\hU \ra \GG_m$ with kernel $U$), if we are prepared to modify the linearisation we can always arrange for this condition to hold. Note that this shift of weights does not change the minimal weight space $Z_{\min}$ and attracting set $X_{\min}$.
\end{rmk}

\begin{defn}[Stabiliser assumptions]\label{def stab ass}
Let $ G = U \rtimes R$ be a group with unipotent radical graded by $\GG_m < Z(R)$ acting linearly on $X = \PP(V)$.
\begin{enumerate}
\item We say \emph{the unipotent stabiliser assumption holds} if 
\begin{equation}\label{condU0}
\dim \Stab_U(z) = 0 \: \text{ for all } \: z \in Z_{\min}, \tag*{$[\hU]_0$}
\end{equation}
\item We say \emph{the reductive stabiliser assumption holds} if
\begin{equation}\label{condR0}
\dim \Stab_{\overline{R}}(z) = 0 \: \text{ for all } z \in Z_{\min}^{\overline{R}-ss}, \tag*{$[\overline{R}]_0$}
\end{equation}
where as $\GG_m$ is central in $R$, there is an induced $R$-action on the $\GG_m$-fixed variety $Z_{\min}$ and we let $Z_{\min}^{\overline{R}-ss}$ denote the GIT semistable locus for $\overline{R}:= R/\GG_m$.
\end{enumerate}
\end{defn}

\begin{rmk}
The unipotent stabiliser assumptions are crucial to apply Proposition \ref{key prop U}, whereas the reductive stabiliser assumptions are used to more readily obtain an explicit Hilbert--Mumford type description of the locus we obtain a quotient of. The reductive stabiliser assumption implies semistability coincides with stability for the $\overline{R}$-action on $Z_{\min}$.

If these stabiliser assumptions fail, then one would like to perform a sequence of equivariant blow-ups to arrange for this to hold on the blow-up (similar to Kirwan's partial desingularisation procedure \cite{Kirwan_desing}) and then construct a quotient of the original scheme using the quotient of the blow-up; however, this procedure is much more complicated in the non-reductive setting described in \cite[$\S$9]{BDHK_proj}. Furthermore, if there are generically positive dimensional unipotent stabilisers, blowing up would only result in constant dimensional stabilisers (rather than trivial stabilisers), and so instead one must filter $U$ by normal subgroups whose stabilisers are constant (assuming this can be done via blow-ups) and then proceed as in \cite[Remark 7.1]{BDHK_proj} and \cite{Qiao}.
\end{rmk}

The strategy is to first use the grading $\GG_m$ to obtain a projective quotient by $\hU = U \rtimes \GG_m$ and then take a reductive GIT quotient by the residual group $\overline{R}=R/\GG_m$. Therefore, we first state the result for quotients by graded unipotent groups $\hU = U \rtimes \GG_m$.

\begin{thm}[The $\hU$-Theorem, {\cite[Theorem 2.16]{BDHK}}]\label{thm Uhat}
Let $ \hU = U \rtimes \GG_m$ be a graded unipotent group acting linearly on $X = \PP(V)$. If the linearised action is adapted and the unipotent stabiliser assumption \ref{condU0} hold, then we have the following statements.
\begin{enumerate}[label=\emph{\roman*)}]
\item There is a geometric $U$-quotient $q_U : X_{\min} \ra X_{\min}/U$ such that $X_{\min}/U$ is a quasi-projective variety.
\item There is a geometric $\hU$-quotient $q_{\hU} : X_{\min} \setminus UZ_{\min} \ra X/\!/\hU:=(X_{\min} \setminus UZ_{\min})/\hU$ such that $X/\!/\hU$ is a projective variety.
\item If the linearised action is {well-adapted}\footnote{See Definition \ref{def well-adapted} below, which can be achieved by further twisting by a rational character.}, then the ring of $\hU$-invariant sections (for an appropriate power) of the linearisation is finitely generated and taking the Proj construction gives the above geometric $\hU$-quotient $q_{\hU}$.
\end{enumerate}
\end{thm}

\noindent Let us outline the structure of the proof.
\begin{enumerate}
\item Since $X_{\min} = \bigcup_{\sigma \in H^0(X,\cO(1))_{\max}} X_{\sigma}$, the $U$-quotient is constructed by applying Proposition \ref{key prop U} to each affine $\hU$-variety $X_{\sigma}$ for $\sigma \in H^0(X,\cO(1))_{\max}$ (see Proposition \ref{prop proof step 1}).
\item Construct a $\GG_m$-equivariant embedding $ X_{\min}/U \hookrightarrow \PP(W)$ with $W := (H^0(X,\cO(r))^U)^*$ for some $r > 0$, to show $ X_{\min}/U$ is quasi-projective (see Proposition \ref{prop proof step 2}). 
\item By appropriately twisting the original linearisation (to make it \emph{well-adapted}), the induced $\GG_m$-action on $\PP(W)$ is adapted and $\PP(W)^{\GG_m-(s)s} = \PP(W)_{\min} \setminus \PP(W_{\min})$; thus
\[ q_{U}^{-1}(\overline{X_{\min}/U}^{\GG_m-(s)s} ) = X_{\min} \setminus U Z_{\min}, \]
has a geometric $\hU$-quotient, as a closed subvariety of $\PP(W)/\!/\GG_m$ (see Proposition \ref{prop proof step 3}).
\item To prove that the ring of $\hU$-invariant sections is finitely generated, one shows that $q_{\hU}$ coincides with an \emph{enveloping quotient} (see \cite[Definition 3.1.6]{BDHK_handbook}) and as the enveloping quotient is projective, the ring of $\hU$-invariant sections for an appropriately divisible power of the well-adapted linearisation is finitely generated by \cite[Corollary 3.1.21]{BDHK_handbook}.
\end{enumerate}
The first two steps give the proof of Theorem \ref{thm Uhat} i), whereas (3) and (4) give statements ii) and iii) respectively; more details on the proofs of Steps (1) - (3) are given in $\S$\ref{sec proof Uhat thm} below.

In (1), it is crucial that the $U$-action is graded by $\GG_m$, that the unipotent stabiliser assumption holds and that the action is adapted in order to apply Proposition \ref{key prop U}. The grading $\GG_m$-action is also used in (2) and (3) to firstly show $X_{\min}/U$ is quasi-projective and then obtain a projective quotient of $X_{\min} \setminus U Z_{\min}$ as a closed subvariety of $\PP(W)/\!/\GG_m$. In (3), in order for the induced $\GG_m$-linearisation on $\PP(W)$ to be adapted, we need the minimal weight $\omega_{\min}$ for the original linearised action to be negative but very small (see the proof of Proposition \ref{prop proof step 3}); this leads to the following notion.

\begin{defn}[Well-adapted linearisation]\label{def well-adapted}
Let $ G = U \rtimes R$ be a group with unipotent radical graded by $\GG_m < Z(R)$ acting linearly on $X = \PP(V)$. We say this linearised action is \emph{well-adapted} if there is $0 < \epsilon <\!< 1$ such that the $\GG_m$-weights on $V$ satisfy
\[ - \epsilon < \omega_{\min} = \omega_0 < 0 < \omega_1 < \dots < \omega_n. \]
\end{defn} 

Let us present a simple example for matrices up to conjugation extending Example \ref{ex conj 2 x 2}.

\begin{ex}[To appear in upcoming joint work with E. Hamilton and J. Jackson]
Consider the upper triangular Borel subgroup $\hU = \GG_a \rtimes_2 \GG_m < \SL_2$ acting by conjugation on $\Mat_{2 \times 2}$. To apply Theorem \ref{thm Uhat}, we consider the projective embedding $\Mat_{2 \times 2} \hookrightarrow X:=\PP(\Mat_{2 \times 2} \oplus k)$, with trivial action on $k$. The minimal weight space is 1-dimensional and spanned by the elementary matrix $E_{21}$, so $Z_{\min} = \{ * \} = \PP(kE_{21})$ is contained at infinity (i.e.\ in $\PP(\Mat_{2 \times 2})$), and $X_{\min} = \{ [A:z] : a_{21} \neq 0 \}$. The unique point in $Z_{\min}$ has trivial $\GG_a$-stabiliser; thus \ref{condU0} holds. By twisting the linearisation to make it adapted, we obtain a geometric $\hU$-quotient
\[ X^{\hU-s} = X_{\min} \setminus U Z_{\min} \ra \PP(1,1,2), \quad [A:z] \mapsto [z: \tr A : \det A ]. \]
As $UZ_{\min}$ is contained at infinity, $\Mat_{2 \times 2} \cap X^{\hU-s} = \Mat_{2 \times 2} \cap X_{\min} = (\Mat_{2 \times 2})_{a_{21}}$ and we obtain a \emph{geometric} $\hU$-quotient of matrices whose bottom left entry is non-zero
\[ (\Mat_{2 \times 2})_{a_{21}} \ra \AA^2, \quad A \mapsto (\tr A,\det A).\]
\end{ex}

Suppose that $G = U \rtimes R$ acts on $X$; then the projective geometric $\hU$-quotient $X/\!/\hU$ of Theorem \ref{thm Uhat} has a residual action of the reductive group $\overline{R} = R /\GG_m$. There is a quotient
\[ X_{\min} \setminus UZ_{\min} \stackrel{q_{\hU}}{\longrightarrow} X/\!/\hU \dashrightarrow X/\!/G:= (X/\!/\hU)/\!/\overline{R},\]
where the second morphism is the reductive GIT quotient; this gives a projective and good $G$-quotient of the open set $q_{\hU}^{-1}((X/\!/\hU)^{\overline{R}-ss})$. The key challenge is to determine this preimage in terms of the original action on $X$. This is easiest when \ref{condR0} holds (i.e.\ semistability coincides with stability for the $\overline{R}$-action on $Z_{\min}$); in this case, we define the following semistable sets.

\begin{defn}[Non-reductive stable set]
Let $ G = U \rtimes R$ be a group with unipotent radical graded by $\GG_m < Z(R)$ acting linearly on $X = \PP(V)$. If the linearisation is well-adapted and both \ref{condU0} and \ref{condR0} hold, then we define the \emph{$G$-stable set} by
\[ X^{G-s}:=X_{\min}^{\overline{R}-ss} \setminus U Z_{\min}^{\overline{R}-ss} = p^{-1}(Z_{\min}^{\overline{R}-ss}) \setminus U Z_{\min}^{\overline{R}-ss}. \]
\end{defn}

\begin{rmk}
Since \ref{condU0} holds for the linear action of the graded unipotent group $\hU$ on $X = \PP(V)$, the $U$-sweep of $Z_{\min}$ is a closed subvariety of $X_{\min}$ by \cite[Lemma 5.4]{BDHK_proj}. Consequently, the $\hU$-stable locus $X^{\hU-s} = X_{\min} \setminus UZ_{\min}$ is open in $X$, and so is the $G$-stable locus.
\end{rmk}

The next result is a special case of \cite[Theorem 2.20]{BDHK_proj} as stated in \cite[Theorem 2.28]{HJ}, where the reductive stabiliser assumption \ref{condR0} is used to describe the preimage of the reductive (semi)stable locus by using the reductive Hilbert-Mumford criterion and comparing the torus weight sets of points $x \in X_{\min}$ and their images under $q_U$. Although some torus weights are lost on applying $q_U$, all torus weights corresponding to maximal grading $\GG_m$-weights survive, and if $x \in U Z_{\min}$ then it has at least one non-minimal weight; these observations together with the assumption that semistability coincides with stability for $\overline{R}$ on $Z_{\min}$ are central in the proof.

\begin{thm}[Construction of non-reductive GIT quotients, {\cite[Theorem 4.28]{BK_momentmap}}]
Let $ G = U \rtimes R$ be a group with unipotent radical graded by $\GG_m < Z(R)$ acting linearly on $X = \PP(V)$. If the linearisation is well-adapted and both \ref{condU0} and \ref{condR0} hold, then there is a projective and geometric $G$-quotient
\[ X^{G-s}=X_{\min}^{\overline{R}-ss} \setminus U Z_{\min}^{\overline{R}-ss}  \longrightarrow X/\!/G,\]
which coincides with the Proj construction associated to the (finitely generated) invariant ring.
\end{thm}

\subsection{Overview of the proof}\label{sec proof Uhat thm}

We provide some details on the proof of Theorem \ref{thm Uhat} i) and ii). The structure of the proof is as follows.

\begin{enumerate}
\item Using Proposition \ref{key prop U}, construct a geometric $U$-quotient of $X_{\min}$ (see Proposition \ref{prop proof step 1}).
\item Construct $X_{\min}/U \hookrightarrow \PP(W)$ to show $X_{\min}/U$ is quasi-projective (see Proposition \ref{prop proof step 2}).
\item Inside $\PP(W)/\!/\GG_m$, construct a geometric $\hU$-quotient of $X^{\hU-s}$ (see Proposition \ref{prop proof step 3}).
\end{enumerate}
The first step relies on the Key Proposition (Proposition \ref{key prop}) about $\GG_a$-slices via $\GG_m$-gradings.

\begin{prop}\label{prop proof step 1}
Let $ \hU = U \rtimes \GG_m$ be a graded unipotent group acting linearly on $X = \PP(V)$. If the linearised action is adapted and \ref{condU0} holds, then there is a geometric $U$-quotient of $X_{\min}$.
\end{prop}
\begin{proof}
Since $X_{\min}$ is the union of the open affine varieties $X_{\sigma}$ over  ${\sigma \in H^0(X,\cO(1))_{\max}}$, we will construct a geometric $U$-quotient by gluing trivial quotients $X_{\sigma} \ra X_{\sigma}/U = \Spec \cO(X_\sigma)^U$ which are constructed by applying Proposition \ref{key prop U} to each affine $\hU$-variety $X_{\sigma}$ for $\sigma \in H^0(X,\cO(1))_{\max}$. 

In fact, by choosing a basis of $V$ consisting of $\GG_m$-weight vectors, which gives an identification $X \cong \PP^n$, it suffices to construct these trivial quotients in the case where $\sigma = x_i \in H^0(X,\cO(1))_{\max}$ is a coordinate function. Then $X_{\sigma} = \PP(V)_{x_i} \cong \AA^n$ has coordinates $x_j/x_i$. Since $x_i \in H^0(X,\cO(1))_{\max}$, its $\GG_m$-weight is $-\omega_{\min}$ (recall that $\omega_{\min}$ is the minimal weight in $V$ and $V = H^0(X,\cO(1))^*$). Hence the weights of the $\GG_m$-action on $X_{\sigma} = \PP(V)_{x_i} \cong \AA^n$ are of the form $\omega_j - \omega_{\min} \geq 0$, where this inequality holds due to the linearised action being adapted. In particular, the flow under $\GG_m$ as $t \ra 0$ exists for all points in $X_{\sigma}$. Since $Z_{\sigma}  \subset Z_{\min}$, the unipotent stabiliser assumption \ref{condU0} implies that the corresponding stabiliser assumption in Proposition \ref{key prop U} holds; hence we obtain the claimed trivial $U$-quotient $X_{\sigma} \ra X_{\sigma}/U = \Spec \cO(X_\sigma)^U$.
\end{proof}

The quotient obtained from this gluing construction is a priori just an abstract scheme, but the next result shows it is in fact quasi-projective. This result is \cite[Lemma 7.6]{BDHK_proj}.

\begin{prop}\label{prop proof step 2}
Let $ \hU $ be a graded unipotent group acting linearly on $X = \PP(V)$ such that the linearised action is adapted and \ref{condU0} holds. There exists a positive integer $r$ and an embedding
\[ X_{\min}/U \hookrightarrow \PP(W)_{\min} \hookrightarrow \PP(W) \]
where $W:=(H^0(X,\cO(r))^U)^*$, the first morphism is a closed immersion and the second morphism is the open inclusion of the minimal attracting set for the induced $\GG_m$-action on $W$.
\end{prop}
\begin{proof}
Fix a basis $\sigma_1,\dots, \sigma_l$ of $H^0(X,\cO(1))_{\max}$ such that $\cO(X_{\sigma_i})$ is finitely generated (see the proof of Proposition \ref{prop proof step 1} above). Then there exists a positive integer $r$, such that for $1 \leq i \leq l$, we have that $R(X,\cO(1))^U_{(\sigma_i^r)}$ is generated by $\{ \frac{f}{\sigma_i^r} : f \in H^0(X,\cO(r))^U \}$. Let $\Sigma_i \in H^0(\PP(W),\cO(1)) \cong H^0(X,\cO(r))$ correspond to $\sigma_i^r$; then $\Sym(W^*)_{\Sigma_i} \ra \cO(X)_{\sigma_i}$ is surjective. 

The inclusion $H^0(X,\cO(r))^U \hookrightarrow H^0(X,\cO(r))$ induces a rational map $\phi : X \dashrightarrow \PP(W)$, which is well-defined on $X_{\min}$, as maximal sections are $U$-invariant, and $\phi|_{X_{\min}} = \overline{\phi} \circ q_U$ for
\[ \overline{\phi} : X_{\min} /U \ra \PP(W). \]
Since $\Spec \cO(X_{\sigma_i})^U$ cover $X_{\min} /U$, we see that $\overline{\phi}$ factors via
\[ \PP(W)_{\min} = \bigcup_{\Sigma \in H^0(\PP(W),\cO(1))_{\max}} \PP(W)_\Sigma. \]
For a partition $\underline{k} = (k_1,\dots,k_l)$ of $r$, the section $\sigma^{\underline{k}}:= \prod_{i=1}^l \sigma_i^{k_i} \in H^0(X,\cO(r))_{\max}$ corresponds to $\Sigma_{\underline{k}} \in H^0(\PP(W),\cO(1))_{\max}$. 

Since being a closed immersion is a local property on the target, $\overline{\phi} : X_{\min} /U \ra \PP(W)_{\min}$ is a closed immersion if $\overline{\phi}_{\underline{k}} : \Spec \cO(X_{\sigma^{\underline{k}}})^U \ra \PP(W)_{\Sigma_{\underline{k}}}$ is a closed immersion for each partition $\underline{k}$, or equivalently $\Sym(W^*)_{\Sigma_{\underline{k}}} \ra \cO(X)_{\sigma^{\underline{k}}}$ is surjective. This last statement is deduced from the fact that $\Sym(W^*)_{\Sigma_i} \ra \cO(X)_{\sigma_i}$ is surjective for each $i$ by the choice of $r$.
\end{proof}

The next two results are described in the discussion after Lemma 7.7 in \cite{BDHK_proj}.

\begin{lemma}\label{lemma proof step 3}
For an adapted linear action of a graded unipotent group $\hU$ on $X = \PP(V)$, assume  \ref{condU0} holds; thus there is a geometric quotient $q_U : X_{\min} \ra X_{\min}/U$, which is locally closed in $\PP(W)$ by Proposition \ref{prop proof step 3}. Let $x \in X_{\min}$; then $q_U \in \PP(W_{\min})$ if and only if $x \in U Z_{\min}$.
\end{lemma}
\begin{proof}
For $q_U(x) \in X_{\min}/U \hookrightarrow \PP(W)_{\min}$, we have that $q_U \in \PP(W_{\min})$ if and only if 
\[ q_U(x) =\lim_{t \ra 0} t \cdot q_u(x) = q_U(\lim_{t \ra 0} t \cdot x) = q_U(p(x)), \]
or, as $q_U$ is a geometric quotient, equivalently $U \cdot x = U \cdot p(x)$, i.e.\ $x \in U Z_{\min}$.
\end{proof}

\begin{prop}\label{prop proof step 3}
For an adapted linear action of a graded unipotent group $\hU$ on $X = \PP(V)$, assume  \ref{condU0} holds. There is a well-adapted rational twist of the $\hU$-linearisation on $X$ such that the induced $\GG_m$-linearisation on $\PP(W)$ is adapted and
\[ \PP(W)^{\GG_m-(s)s} = \PP(W)_{\min} \setminus \PP(W_{\min}).\]
Furthermore, the preimage under $q_U$ of the $\GG_m$-stable locus of the closure of $X_{\min}/U$ in $\PP(W)$
\[ q_U^{-1}(\overline{X_{\min}/U}^{\GG_m-(s)s}) = X_{\min} \setminus U Z_{\min} \]
admits a projective geometric $\hU$-quotient $X_{\min} \setminus U Z_{\min} \ra (\overline{X_{\min}/U})/\!/\GG_m$.
\end{prop}
\begin{proof}
Recall that $\omega_{\min}$ is the minimal weight on $V=H^0(X,\cO(1))^*$ and so $r \omega_{\min}$ is the minimal weight on $W:=(H^0(X,\cO(r))^U)^*$. Pick $\epsilon >0$ so that $r \epsilon < 1$. Let $\chi : \hU \ra \GG_m$ be the rational character corresponding to $\omega_{\min} + \epsilon \in \QQ$. Then the $\GG_m$-weights $\alpha_j$ on $\cO_{\PP(W)}(1)^{r\chi}$ satisfy
\[ \alpha_{\min} = \alpha_0 = r \omega_{\min} - r \chi = - r \epsilon < 0 < 1 - r \epsilon = r \omega_{\min} + 1 - r \chi \leq  \alpha_1 < \cdots < \alpha_{\max} \]
so the induced $\GG_m$-linearisation $\cO_{\PP(W)}(1)$ is adapted and $\PP(W)^{\GG_m-(s)s} = \PP(W)_{\min} \setminus \PP(W_{\min})$. 
This linearisation on $\PP(W)$ is induced from the well-adapted twisted linearisation $\cO_X(1)^\chi$. The final claim follows from Lemma \ref{lemma proof step 3}, as $\overline{X_{\min}/U}^{\GG_m-(s)s} = (X_{\min} /U) \setminus ((X_{\min} /U) \cap \PP(W_{\min}))$.
\end{proof}

The proof of the finite generation of an appropriate power of the linearisation is given in the discussion proceeding Corollary 7.10 in \cite{BDHK_proj}.

\section{Recent applications of non-reductive GIT}

In this section we will give an overview of some recent applications of non-reductive GIT.

\subsection{Moduli of jets of map germs and hyperbolicity}\label{sec hyperbolicity}

B\'{e}rczi and Kirwan \cite{BK_hyperbolicity} used non-reductive GIT to construct and study compactifications of spaces of invariant jet differentials in order to prove polynomial versions of the Green--Griffiths--Lang conjecture and Kobayashi conjecture concerning hyperbolicity properties of generic smooth projective hypersurfaces. Let us outline these conjectures and the approach using non-reductive GIT.

A complex projective manifold $X$ is \emph{Brody hyperbolic} if every holomorphic map $f : \CC \ra X$ is constant. For example, in dimension $1$, a curve is hyperbolic if and only if $g \geq 2$. Hyperbolic varieties are interesting from the point of view of complex geometry and also for their conjectural Diophantine properties (Lang conjectured that if a projective variety defined over $\QQ$ is hyperbolic, then $X(\QQ)$ is finite).

The Kobayashi conjecture predicts that a very general hypersurface $X \subset \PP^{n+1}$ of sufficiently large degree $d_n$ is Brody hyperbolic. Green, Griffiths and Lang conjectured that every projective algebraic variety $X$ of general type is \emph{weakly hyperbolic}; that is, there exists a proper subvariety $Y \subsetneq X$ such that the image of every holomorphic map $f : \CC \ra X$ is contained in $Y$. These conjecture are related by a recent result of Riedl and Yang \cite{RiedlYang}: if the Green--Griffiths--Lang conjecture holds for projective hypersurfaces of dimension $n$ and degree at least $d_n$, then the Kobayashi conjecture is true for projective hypersurfaces of dimension $n$ with degree at least $d_{2n-1}$. The strategy for approaching these conjectures goes back to work of Demailly \cite{Demailly} and Siu \cite{Siu}, which involves studying \emph{invariant jet differentials}; here non-reductive group actions naturally arise as reparametrisation groups.

For a smooth projective complex variety $X$ of dimension $n$, the bundle of $J_kX \ra X$ of \emph{$k$-jet germs in $X$} has fibre over $p \in X$ is given by germs of holomorphic maps $f : (\CC,0) \ra (X,p)$ for fixed local coordinates at $p$ up to the equivalence relation given by equality of the first $k$-derivatives at $0$; thus the fibres can be represented by truncated Taylor expansions or equivalently $k$-tuples of vectors in $\CC^n$ given by the first $k$-derivatives. The transition functions are polynomial, but not linear, so $J_kX \ra X$ is not a vector bundle. The group $\mathrm{Diff}_k$ of regular $k$-jets of maps $(\CC,0) \ra (\CC,0)$ acts fibrewise on $J_kX \ra X$ by reparametrisations;  $\mathrm{Diff}_k$ is naturally an upper triangular subgroup of $\GL_k$ and fortunately is a graded unipotent group
\[ \mathrm{Diff}_k \cong U_k \rtimes \CC^*, \]
where $\dim U_k = k-1$. Green and Griffiths studied algebraic differential operators, which are polynomial functions on $J_kX$, and constructed a sheaf of algebraic differential operators of order $k$ of fixed weighted degree (with respect to the $\CC^*$-weights). Demailly considered a subbundle of jet differentials invariant under reparametrisations from $U_k$. A key tool to finding invariant jet differentials is to produce a projective completion of the fibrewise quotient of $\mathrm{Diff}_k$ acting on the jet bundle $J_kX \ra X$. The projective completion given by B\'{e}rczi and Kirwan \cite{BK_hyperbolicity} uses non-reductive GIT, where a blow-up at $Z_{\min}$, which is just a point, is needed for the unipotent stabiliser assumption to hold. They then use intersection theory for non-reductive GIT quotients (see \cite{BK_momentmap} and $\S$\ref{sec sympl nrGIT} below) to prove a polynomial version of the Green--Griffiths--Lang conjecture.

\begin{thm}[Polynomial Green--Griffiths--Lang Theorem of B\'{e}rczi--Kirwan \cite{BK_hyperbolicity}]
A generic smooth projective hypersurface of dimension $n$ and degree $d \geq 32n^4$ is weakly hyperbolic.
\end{thm}

By work of Riedl--Yang \cite{RiedlYang}, this gives a polynomial Kobayashi theorem \cite[Theorem 1.4]{BK_hyperbolicity} for generic smooth projective hypersurfaces of dimension $n$ and degree $d \geq 32(2n-1)^4$.

\subsection{Moduli spaces of hypersurfaces in weighted projective orbifolds}\label{sec weighted proj hyp}

One classical application of reductive GIT is to construct moduli spaces of projective hypersurfaces as the GIT quotient of the $\PGL_{n+1}$-action on $\PP(k[x_0,\dots,x_n]_d)$; this gives compactifications of moduli spaces of smooth hypersurfaces, as Mumford proved that any smooth hypersurface $X \subset \PP^n$ of degree $d \geq 3$ is GIT stable when $n > 1$. In general determining precisely which other hypersurfaces are (semi)stable is challenging, even with the Hilbert--Mumford criterion in hand.

The advent of non-reductive GIT enables this to be extended to hypersurfaces in weighted projective spaces and more general projective toric varieties, whose automorphism groups are non-reductive affine algebraic groups and are explicitly described by the work of Cox as quotients of the graded automorphism group of the Cox ring.

\begin{ex}
The weighted projective plane $\PP(1,1,2)$ has automorphism group given by
\[ 1 \lra \GG_m \stackrel{(tI_2,t^2,0,0,0)}{\lra} (\GL_2 \times \GG_m ) \rtimes \GG_a^3 \lra \Aut(\PP(1,1,2)) \lra 1 \]
where the unipotent group appears from automorphisms of the form $z \mapsto z + ax^2 + bxy + c y^2$.

\end{ex}

Fortunately the automorphism groups of weighted projective spaces have graded unipotent radicals; see \cite[Lemma 4.1]{BDHK} and also \cite{Bunnett}. 

Bunnett \cite{Bunnett} studied the application of non-reductive GIT to moduli of weighted projective hypersurfaces and more generally hypersurfaces in toric orbifolds. In a well-formed weighted projective space $\PP(a_0,\dots,a_n)$ he proves \cite[Theorem 5.18]{Bunnett} that any \emph{quasi-smooth} hypersurface (see \cite[$\S$3]{BatyrevCox}) of degree $d \geq 2 + \max\{a_0,\dots,a_n\}$ is stable (in the sense of non-reductive GIT) provided the unipotent stabiliser assumption holds, so that no blow-ups are required. 

In the case of a well-formed weighted projective space $\PP(a_0,\dots,a_n)$, any quasi-smooth hypersurface of degree $d \geq 2 + \max\{a_0,\dots,a_n\}$ has finitely many automorphisms coming from the ambient automorphisms of $\PP(a_0,\dots,a_n)$ by \cite[Theorem 3.13]{Bunnett}. Consequently, the Keel--Mori Theorem gives the existence of a coarse moduli spaces as an algebraic space. However, non-reductive GIT gives the construction of a quasi-projective moduli space.

For a toric variety $X$, there is an $A$-discriminant (see  \cite{GKZ}) for hypersurfaces of class $\alpha$, which vanishes on non-quasi-smooth hypersurfaces (in contrast to the case of projective hypersurfaces, the converse is not necessarily true, as the $A$-discriminant only checks for singularities in the sweep under $G = \Aut_\alpha(X)$ of the torus $T \subset X$, see \cite[Remark 4.11]{Bunnett}). The $A$-discriminant of \cite{GKZ} is interpreted as an invariant section of a twisted linearisation in \cite[Corollary 4.12]{Bunnett}.

Let us explain the non-reductive GIT set-up for hypersurfaces in $\PP(a_0,\dots,a_n)$ of degree $d$. Assume that hypersurfaces of degree $d$ are Cartier divisors (i.e.\ the lowest common multiple of the weights divides $d$). Consider the non-reductive group $G=\Aut(\PP(a_0,\dots,a_n))$ acting on the space $X = \PP(k[x_0,\dots,x_n]_d)$ of weighted degree $d$ homogeneous polynomials. Bunnett proves that quasi-smooth hypersurfaces are contained in the $\hU$-stable set $X_{\min} \setminus U Z_{\min}$ (under the unipotent stabiliser assumption). Using the non-reductive GIT Hilbert--Mumford criterion of \cite{BDHK_proj}, he shows that quasi-smooth hypersurfaces are stable for the action of $G$ assuming that $d$ is a Cartier degree with $d \geq 2 + \max\{a_0,\dots,a_n\}$ and the unipotent stabiliser assumption holds. Furthermore, if the weighted projective space has only two weights, then the unipotent stabiliser assumption holds (see \cite[Proposition 5.9]{Bunnett}).

Bunnett obtains the best results for Cartier hypersurfaces in a rational cone $\PP(1,\dots,1,r)$ of degree $d \geq r +2$ (see \cite[Theorem 5.20]{Bunnett}): he explicitly describes the quasi-smooth locus as the non-vanishing locus of a section (which is obtained by multiplying the $A$-discriminant with a variable) and constructs a $U$-quotient of this open affine variety using Proposition \ref{key prop U}, where the necessary unipotent stabiliser assumption is easily verified. He then directly obtains a geometric quotient of the locus of quasi-smooth projective hypersurfaces, which is a projective over affine variety, because it is constructed as a reductive GIT quotient of the affine $U$-quotient twisted by a character as in $\S$\ref{sec affine GIT twisted char} rather than using the more complicated methods of \cite{BDHK_proj}.

\subsection{Moduli of unstable objects}\label{sec quotient unstable strata}

Recall from $\S$\ref{sec instability} that associated to a linear action of a reductive group $G$ on $\PP(V)$ and a choice of norm, there is an instability stratification 
\begin{equation}\label{instab strat 2}
 \PP(V) = \bigsqcup_{\beta \in \cB} S_\beta 
 \end{equation}
where $S_\beta \cong G \times^{P_\lambda} Y_\beta^{ss}$ for a parabolic subgroup $P_\lambda < G$. A categorical $P_\lambda$-quotient of $Y_\beta^{ss}$, or equivalently a categorical $G$-quotient of $S_\beta$, is given by Proposition \ref{prop cat quotient unstable strata}; however, as explained after this proposition, this is far from being an orbit space as it factors via the retraction $p_\beta : Y_\beta^{ss} \ra Z_\beta^{ss}$ sending a point to its flow under $\lambda$ as $t \ra 0$. Instead, we would like to apply non-reductive GIT to the action of $P_\lambda$ on the closure $\overline{Y_\beta}$ of $Y_\beta \subset X$, where we can twist the linearisation by (a rational multiple of) a character corresponding to $\lambda$ to make it well-adapted. Fortunately, the non-reductive notion of stability precisely picks out the locus we would like and removes (the $P_\lambda$-sweep) of the limit set $Z_{\beta}^{ss}$; see Theorem \ref{thm nrGIT unstable strata} below.

For the parabolic group $P_\lambda = U_\lambda \rtimes L_\lambda$ acting on the blade closure $X=\overline{Y_\beta}$ of an unstable stratum $S_\beta$ as in \eqref{instab strat 2} above, there is a twisted rational linearisation $\cL_{\beta(1+ \epsilon)}$ which is well-adapted. Furthermore, we have that in the non-reductive GIT notation the map $p : X_{\min} \ra Z_{\min}$ coincides with the retraction $p_\beta : Y_\beta \ra Z_\beta$ appearing in the description of the unstable strata. Furthermore, $Z_\beta^{ss}$ is defined to be the semistable locus for $L_\lambda$ with respect to $\cL_\beta$, or equivalently for $\overline{L_\lambda}:=L_\lambda /\lambda(\GG_m)$ as $\lambda(\GG_m)$ acts trivially, which coincides with the semistable locus in $Z_{\min}$ appearing in the definition of the non-reductive stable locus. 

\begin{thm}[Non-reductive GIT quotients of unstable strata, {\cite[Theorem 1.1]{HJ}}]\label{thm nrGIT unstable strata}
For the parabolic group $P_\lambda = U_\lambda \rtimes L_\lambda$ graded by $\lambda$ acting on the blade closure $X:=\overline{Y_\beta}$ of an unstable stratum $S_\beta$ as in \eqref{instab strat 2} with the well-adapted linearisation  $\cL_{\beta(1+ \epsilon)}$, the following statements hold.
\begin{enumerate}[label=\emph{\roman*)}]
\item If \ref{condU0} holds, then there is a projective geometric $\hU_\lambda$-quotient
\[ q_{\hU_\lambda} : \overline{Y_\beta}^{\hU_\lambda-s}= Y_\beta \setminus U Z_\beta \lra \overline{Y_\beta}/\!/ \hU_\lambda \]
and by taking a reductive GIT quotient by $\overline{L_\lambda}$ one obtains a projective categorical $P_\lambda$-quotient of an open subset of the $\hU$-stable locus.
\item If both \ref{condU0} and \ref{condR0} hold, then there is a projective geometric $P_\lambda$-quotient
\[ q_{P_\lambda} : \overline{Y_\beta}^{P_\lambda-s}= Y_\beta^{ss} \setminus U Z_\beta^{ss} \lra \overline{Y_\beta}/\!/ P_\lambda. \]
Moreover, the ring of invariant sections if finitely generated and $q_{P_\lambda}$ coincides with the Proj construction for this invariant ring.
\end{enumerate}
\end{thm}

We would like to apply this theorem to moduli of objects in an abelian category, where there are moduli-theoretic instability filtrations, such as the Harder--Narasimhan (HN) filtrations for vector bundles \cite{HN}; for example, moduli of sheaves on  projective schemes or moduli of quiver representations. In these examples, the GIT instability stratification has been compared with the moduli-theoretic Harder--Narasimhan stratification \cite{GSZ,hoskins_kirwan,hoskins_affinestrat,hoskins_stackstrat,zamora} and this suggests moduli of objects of fixed HN type should be constructed as non-reductive GIT quotients.

Unfortunately the stabiliser assumptions in Theorem \ref{thm nrGIT unstable strata} are quite restrictive and so it is only possible in quite limited situations. For example, for vector bundles (or Higgs bundles) on a smooth projective curve of fixed HN type, the reductive stabiliser assumption \ref{condR0} only holds for coprime HN types of length $2$ (i.e.\ the HN filtration has two terms and the invariants for the successive quotients are coprime, so that semistability coincides with stability) and even in this case, the unipotent stabiliser assumption rarely holds and so blow-ups are needed (see \cite[$\S$3.2.1]{HJ} for a detailed discussion). In this length 2 coprime case, the non-reductive GIT quotient picks out non-split HN filtrations of length $2$ whose automorphism groups have a fixed dimension; see \cite{BPRios,Jackson} for the case of vector bundles and \cite{Hamilton} for the case of Higgs bundles. To rectify the failure of the reductive stabiliser assumption \ref{condR0}, one can alternatively perform a quotient in stages, using different 1-PSs in the centre of $L_\lambda$ to grade different subgroups of the unipotent radical as in \cite{HJ}; this results in a natural notion of stability for sheaves of a fixed HN type, but again the unipotent stabiliser assumption is rarely satisfied, so a blow-up procedure would be required.

\subsection{Interactions with symplectic geometry and cohomological descriptions}\label{sec sympl nrGIT}

GIT quotients for complex reductive groups are closely related to symplectic quotients for a maximal compact group (see Remark \ref{rmk sympl GIT}); the close relationship between the reductive GIT instability stratification and a Morse-theoretic stratification for the norm square of the moment map was used in \cite{kirwan} to describe the rational Betti numbers of reductive GIT quotients. 

Fortunately, for non-reductive groups with internally graded unipotent radicals, this close relationship with symplectic geometry has been extended by work of B\'{e}rczi and Kirwan \cite{BK_momentmap}, and applied to compute cohomology of non-reductive GIT quotients.

For a reductive group $G$ acting on a smooth complex projective variety $Y$, to construct a moment map one fixes a maximal compact subgroup $K < G$ and a symplectic form $\omega$ invariant under the $K$-action. The moment map for this maximal compact and symplectic form is a $K$-invariant map $\mu_{K,\omega} : Y \ra \fK^*:= \Lie(K)^*$ with the moment map property (that it lifts the infinitesimal action via the correspondence between vector fields and forms given by $\omega$). However, any other maximal compact subgroup is of the form $g^{-1} K g$ and $g^*\omega$ is invariant under the $g^{-1} K g$-action with moment map $\mu_{g^{-1}Kg,g^*\omega} = \mathrm{Ad}^*_{g^{-1}} \circ \mu_{K,\omega} \circ g$. Therefore rather than defining a moment map $\mu_{K,\omega} : Y \ra \fK^*$, B\'{e}rczi and Kirwan instead fix a \emph{$G$-equivariant K\"{a}hler structure} $\Omega$ (namely a $G$-orbit in the space  of pairs $(K,\omega)$ of maximal compact subgroups of $G$ and K\"{a}hler forms on $Y$ which are invariant under this maximal compact) and define an \emph{$\Omega$-moment map} to be a smooth $G$-equivariant map
\[ m_{G,Y,\Omega} : \Omega \times Y \ra \fg^*\]
such that $m_{G,Y,\Omega}(K,\omega,-) = \iota_K \circ \mu_{K,\omega} : Y \ra \fK^* \hookrightarrow \fg^*$ is a moment map for the $K$-action on $(Y,\omega)$, where as $\fg = \fK \otimes \CC$, we have a canonical embedding $\iota_K : \fK^* \hookrightarrow \fg^*$.

Let us explain how B\'{e}rczi and Kirwan define moment maps for a smooth complex projective variety with an action of a graded unipotent group $\hU = U \rtimes \CC^*$. Assume that $\hU < G$ is a subgroup of a reductive group and that $X \subset Y$ is a submanifold of a compact K\"{a}hler manifold $Y$ with a $G$-action on $Y$ that restricts to the given $\hU$-action on $X$ (note that $X$ is not required to be invariant under the $G$-action). As above, fix a $G$-equivariant K\"{a}hler structure $\Omega$ on $Y$ and let $m_{G,Y,\Omega} : \Omega \times Y \ra \fg^*$ be an $\Omega$-moment map, they define $m_{\hU,X,\Omega} : X \times \Omega \ra \hat{\fu}^*$ by restricting the $\Omega$-moment map to $X$ and composing with the restriction $\fg^* \ra \hat{\fu}^*$
\[ \xymatrix{  \Omega \times X \ar[r] \ar[rd]_{m_{\hU,X,\Omega}}  & \fg^* = \fK^* \oplus i \fK^* \ar[d] \\ & \hat{\fu}^* = \RR \oplus i \RR \oplus {\fu}^*.  } \]

Assuming the unipotent stabiliser assumption \ref{condU0} holds, B\'{e}rczi and Kirwan provide a moment map description of the non-reductive GIT quotient: for any $(K,\omega) \in \Omega$, they show
\[ X^{\hU-s} = \hU (\mu_{(K,\omega)}^{\hU})^{-1}(0), \quad \text{where} \quad \mu_{(K,\omega)}^{\hU}:= m_{\hU,X,\Omega}( (K,\omega),-) : X \ra \hat{\fu}^* \] 
and that $0$ is a regular value of $\mu_{(K,\omega)}^{\hU}$ and the inclusion of the zero level set of the moment map in the stable locus induces a diffeomorphism of orbifolds
\[ (\mu_{(K,\omega)}^{\hU})^{-1}(0)/(K \cap \hU) = (\mu_{(K,\omega)}^{\hU})^{-1}(0) / S^1 \simeq X^{\hU-s}/U = X/\!/ \hU, \]
which can be viewed as a non-reductive Kempf-Ness Theorem. They extend this result to an action of $H=U \rtimes R$ with internally graded unipotent radical (see \cite[Theorem 1.1]{BK_momentmap}).

They apply this to compute Betti numbers of non-reductive GIT quotients. Assuming \ref{condU0} holds so that $X^{\hU-s} = X_{\min} \setminus U Z_{\min}$, they show that the stratification $X_{\min} = X^{\hU-s} \sqcup U Z_{\min}$ is $\hU$-equivariantly perfect and so the Poincar\'{e} series of $ X/\!/ \hU$ can be computed from that of $Z_{\min}$. Similarly for  $H=U \rtimes R$ as above, they show the Poincar\'{e} series of $X/\!/H$ can be computed from that of the reductive GIT quotient of $Z_{\min}$ by $\overline{R}:= R/\GG_m$, whose Poincar\'{e} series can in turn be computed as in \cite{kirwan} using the reductive GIT instability stratification of $\S$\ref{sec instability}. 

Furthermore, they adapt methods of Martin \cite{Martin} relating the rational cohomology of GIT quotients by reductive groups to that of GIT quotients for a maximal torus. Martin shows the intersection pairing for the reductive GIT quotient can be computed from that of the GIT quotient for the maximal torus via an integration formula, which can then be combined with torus localisation techniques. In the non-reductive case, this enables a description of the rational cohomology ring and a non-reductive integration formula \cite[Theorems 1.4 and 1.5]{BK_momentmap}, which leads to a residue formula for the intersection pairing on the non-reductive GIT quotient.

These methods are used to prove the polynomial versions of the Green--Griffiths--Lang conjecture and Kobayashi conjecture described in $\S$\ref{sec hyperbolicity}. They can also be applied in the future to describe the cohomology of new moduli spaces constructed via non-reductive GIT.

\medskip

\bibliographystyle{amsplain}
\bibliography{references}

\medskip \medskip

\noindent{Radboud University, IMAPP, PO Box 9010,
6525 AJ Nijmegen, Netherlands} 

\medskip \noindent{\texttt{v.hoskins@math.ru.nl}}

\end{document}